\documentclass[11pt]{article}

\usepackage{dsfont}
\usepackage{amsmath, amsthm, amssymb}
\usepackage{fullpage}
\usepackage{xypic}
\usepackage{graphicx}
\usepackage{color}
\usepackage[latin1]{inputenc}
\usepackage{indentfirst}

\newtheorem{theorem}{Theorem}[section]
\newtheorem*{theorema}{Theorem A}
\newtheorem*{theoremb}{Theorem B}
\newtheorem*{theoremc}{Theorem C}
\newtheorem{corollary}[theorem]{Corollary}

\newtheorem{proposition}[theorem]{Proposition}
\newtheorem{lemma}[theorem]{Lemma}
\newtheorem{definition}[theorem]{Definition}

\newcommand{\gim}{\mathrm{gim}}
\newcommand{\gker}{\mathrm{gker}}
\newcommand{\trace}{\mathrm{trace}}
\renewcommand{\ker}{\mathrm{ker}}
\newcommand{\im}{\mathrm{im}}
\newcommand{\Fix}{\mathrm{Fix}}

\renewcommand{\int}{\mathrm{int}}
\newcommand{\Inv}{\mathrm{Inv}}

\title{\textbf{About the homological discrete Conley index of isolated invariant acyclic continua}}

\author{Luis Hern\'andez-Corbato, Patrice Le Calvez, Francisco R. Ruiz del Portal \thanks{ The authors have been supported by MICINN,
MTM 2009-07030. This work is part of the first author Ph. D. Thesis, who has been also supported by a FPU scholarship, AP2008-00102.
}}

\begin{document}

\maketitle

\begin{abstract}
This article includes an almost self-contained exposition on the discrete Conley index and its duality.
We work with a locally defined homeomorphism $f$ in $\mathds{R}^d$ and an acyclic continuum $X$, such as a cellular set or a fixed point,
invariant under $f$ and isolated.
We prove that the trace of the first discrete homological Conley index of $f$ and $X$
is greater than or equal to -1 and describe its periodical behavior.
If equality holds then the traces of the higher homological indices are 0.
In the case of orientation-reversing homeomorphisms of $\mathds{R}^3$, we obtain a characterization of the fixed point
index sequence $\{i(f^n, p)\}_{n \ge 1}$ for a fixed point $p$ which is isolated as an invariant set.
In particular, we obtain that $i(f,p) \leq 1$.
As a corollary, we prove that there are no
minimal orientation-reversing homeomorphisms in $\mathds{R}^3$.
\end{abstract}


\section{Introduction}

Our work deals with a local study of isolated invariant sets of maps defined in a Euclidean space.
An invariant set $X$ is \emph{isolated} if it is the maximal invariant subset contained in a neighborhood of $X$.
In our setting, all isolated invariant sets are compact. In addition, we restrict
our considerations to homeomorphisms, which will be defined in an open subset of $\mathds{R}^d$.
These maps will be called \emph{local homeomorphisms}.
The local nature of our work
makes all the results presented in the article equally valid if we replace $\mathds{R}^d$ by an arbitrary manifold.
Despite a study of the dynamics
of a map around a fixed point was at first the objective of this work, for most of our results
we simply assume that the compact isolated invariant set
is connected and \emph{acyclic}, i.e., all its \v{C}ech rational homology groups are trivial for $r \ge 1$.
Of course, fixed points and \emph{cellular} compacta, that is, compact sets having a basis of neighborhoods
composed of closed balls, or more generally trivial shape continua (see \cite{MardesicSegal} for information about shape theory) are particular cases of acyclic continua. A remarkable fact, see \cite{gabites}, is that a trivial shape continuum of a closed 3-manifold is an isolated invariant set for a flow if and only if it is cellular.
This question is open in the discrete case.

An important and successful topological invariant used in dynamical systems
is the Conley index.
For a complete introduction to the theory we refer to \cite{handbookconley}.
In this work, the dynamics is generated by the iteration of a map, it
is discrete. Our considerations mainly focus in the homological discrete Conley index,
which will be here provided with rational coefficients.
It associates to an isolated invariant set $X$ and a map $f$ an equivalence class, $\textit{h}(f, X)$, of graded
linear endomorphisms of vector graded spaces over $\mathds{Q}$
 which contains, up to conjugation, exactly one
automorphism. If we restrict our considerations to grade $r$, we obtain
the $r$-homological discrete Conley index, $\textit{h}_r(f, X)$.
As will be shown later, all the maps contained in the equivalence class
$\textit{h}_r(f, X)$ have equal traces, so the obvious definition $\trace(\textit{h}_r(f, X))$ makes sense
and, as a matter of fact, a great part of the contents of this article deals with this invariant.

Another very important invariant with which this work is concerned is fixed point index,
which associates to a map $f$ and a point or, more generally,
a set $X$ an integer $i(f, X)$.
It illustrates the aim of many topological invariants appearing in dynamical systems,
if the fixed point index of a set is non-zero then it contains fixed points.
The fixed point index $i(f, X)$ is an algebraic measure of the set of fixed points of a map $f$ in $X$.
The definition requires the set $X$ to have
a small neighborhood $U$ such that $\Fix(f) \cap (U \setminus X) = \emptyset$,
which will be always the case if $X$ is an isolated invariant set.
Then, the fixed point index is defined as the Brouwer degree of the map $\mathrm{id} - f$ restricted to $U$.
We will approach this invariant through
the more powerful tools of Conley index. Fixed point index is coarser than
discrete homological Conley index, as the following formula shows:
\begin{equation}\label{eq:lefschetzlike}
i(f, X) = \sum_{r \ge 0}(-1)^r \trace(\textit{h}_r(f, X)).
\end{equation}
This Lefschetz-like formula provides information about the fixed point index
of a compact isolated invariant set once we estimate the traces of its discrete
homological Conley indices.

An isolated invariant set $X$ of a local homeomorphism $f$ is also invariant
under $f^n$, for any positive integer $n$. It is immediate to prove that $X$ is isolated for $f^n$
as well. One may wonder about the relationship between the homological Conley indices
$\textit{h}(f, X)$ and $\textit{h}(f^n, X)$. Not surprisingly, the $n$-th power
of an endomorphism contained in the class $\textit{h}(f, X)$ belongs to the
equivalence class $\textit{h}(f^n, X)$. Therefore, once we obtain
an element of $\textit{h}(f, X)$, hence of every $\textit{h}_r(f, X)$ for $r \ge 0$,
we can use equation (\ref{eq:lefschetzlike}) to compute, not only the fixed point
index $i(f, X)$, but the fixed point index of any positive iterate of $f$ at $X$, $i(f^n, X)$.
This generalization of equation (\ref{eq:lefschetzlike}) is proved
in Subsection \ref{subsec:indexcomputation} and reads as follows,
\begin{equation}\label{eq:lefschetzlikeiterates}
i(f^n, X) = \sum_{r \ge 0}(-1)^r \trace(\textit{h}_r(f^n, X)).
\end{equation}

An invariant set $X$ of a local homeomorphism $f$ of $\mathds{R}^d$
 is an \emph{attractor} (resp. a \emph{repeller}) if it is the maximal
backward (forward) invariant subset contained in a neighborhood of $X$.
The discrete homological Conley indices are very easy to describe in these cases,
provided that $X$ is an acyclic continuum.
If $X$ is an attractor, for any positive integer $n$, $\textit{h}_0(f^n, X)$ is represented by the identity map over
$\mathds{Q}$ and all higher homological indices $\textit{h}_r(f^n, X)$ are trivial, they
contain the zero automorphism.
In the repeller case, the only non-trivial index is
$\textit{h}_{d}(f^n, X)$ and it is represented by the map
$s : \mathds{Q} \to \mathds{Q}$ defined by $s(x) = d(f)^n x$,
where $d(f)$ equals 1 if $f$ preserves orientation and $-1$ if $f$ reverses orientation.
The details of this description can be found in Subsection \ref{subsec:attractorrepeller}.
Equation (\ref{eq:lefschetzlikeiterates}) then shows that the fixed point index, $i(f^n, X)$, of
an attractor and a repeller are $1$ and $(-1)^d d(f)^n$, respectively, as is well-known.

The first result presented in the article deals with the 1-homological index.

\begin{theorema}
Let $f$ be a local homeomorphism of $\mathds{R}^d$ and $X$ an isolated invariant acyclic continuum.
The trace of the first homological discrete Conley index satisfies
$$\trace(\textit{h}_1(f, X)) \ge -1.$$
Furthermore, there exists a finite set $J$ and a map $\varphi : J \to J$ such that, for $n \ge 1$,
$$\trace(\textit{h}_1(f^n, X)) = -1 + \#\Fix(\varphi^n).$$
\end{theorema}

The proof of this theorem follows from a combinatorial description of the first
homological discrete Conley index, which will be expressed in terms of a map $\varphi$ defined over the finite
set of connected components of a neighborhood of the exit set of an isolating
neighborhood of $X$. The precise statement is the first part of Theorem \ref{thm:key}
 of Subsection \ref{subsec:theorems}, some combinatorial ideas are
 essential part of its proof, which is the content of Section \ref{sec:proofthm}.
For higher dimensional homological indices the intuition would be to think that homology
classes are permuted when $\varphi$ does not fix any element. This idea will be formalized as follows:
there exists, for every $r > 1$, a decomposition in direct summands of a vector space which can be used to compute the
$r$-homological discrete Conley index such that the index map
permutes the summands in an equivalent way as $\varphi$ does. Then, we obtain
the following theorem:

\begin{theoremb}
Let $f$ be a local homeomorphism of $\mathds{R}^d$ and $X$ an isolated invariant acyclic continuum.
If $\trace(\textit{h}_{1}(f, X)) = -1$ then $\trace(\textit{h}_{r}(f, X)) = 0$ for any $r > 1$.
\end{theoremb}

The previous two theorems focus in the behavior of the first homological index.
One deduces from Theorem A and equation (\ref{eq:lefschetzlikeiterates}) that
$i(f^n, X) = 1 - \#\Fix(\varphi^n)$ in dimension $d = 2$, which was already known,
see \cite{lecalvezyoccozconley} or \cite{rportalsalazar}.
In principle, no further applications can be carried into higher dimensions.
A bit of help in this direction is provided
by a duality result between Conley indices. It says that, given an isolated invariant set
$X$ of a local homeomorphism $f$ of $\mathds{R}^d$, certainly also isolated
as an invariant set for $f^{-1}$, for any $0 \le r \le d$,
the $(d-r)$-index of $X$ and $f$ is dual, up to sign, to the $r$-index of $X$ and $f^{-1}$:
\begin{equation}
\tag{Szymczak's duality}
\textit{h}_{d-r}(f, X) \cong d(f) \cdot (\textit{h}_r(f^{-1}, X))^*.
\hspace{-2cm}
\end{equation}
The sign $d(f)$ is $-1$ or $1$ depending on whether $f$ reverses or preserves orientation.
This duality was first stated by Szymczak in \cite{szymczak} for the discrete Conley index. For the sake of completeness,
we include in the article a short proof of this duality together with a new point of view which makes
it closer to the original Conley index, which was defined for flows.
The reader may check that the homological indices of attractors and repellers,
which have been previously described, agree
with Szymczak's duality. Note that a repeller is an attractor for the map $f^{-1}$.

As a corollary of this duality and Theorems A and B, we obtain the following inequality in dimension 3.

\begin{corollary}\label{cor:indexinequality}
Suppose that $f$ is a local orientation-reversing homeomorphism of $\mathds{R}^3$
and $X$ is an isolated invariant acyclic continuum. Then,
$$i(f, X) \le 1.$$
\end{corollary}
\begin{proof}
In the case $X$ is either an attractor or a repeller,
the fixed point index is equal to $1$ or $-d(f)$, respectively, and the inequality holds.
Otherwise, in Subsection \ref{subsec:attractorrepeller} we will show that all $r$-indices are trivial for $r = 0$ and $r \ge 3$.
Furthermore, Szymczak's duality and Theorem A yield that $\trace(\textit{h}_2(f, X)) = - \trace(\textit{h}_1(f^{-1}, X)) \le 1$.
It follows from (\ref{eq:lefschetzlike}) that
$$i(f, X) = - \trace(\textit{h}_1(f, X)) + \trace(\textit{h}_2(f, X)) \le 2.$$
Theorem B shows that this bound is never attained, if the first term is $1$ then the second vanishes. Therefore, we
conclude that $i(f, X) \le 1$.
\end{proof}

Another easy corollary of Theorems A and B gives a sufficiency condition for existence of fixed points in terms
of the trace of the homological discrete Conley index.

\begin{corollary}
Let $f$ be a local homeomorphism of $\mathds{R}^d$ and $X$ an isolated invariant acyclic continuum. If (at least)
one of the following alternatives hold:
\begin{itemize}
\item $\trace(\textit{h}_{1}(f, X)) = -1$,
\item $f$ preserves orientation and $\trace(\textit{h}_{d-1}(f, X)) = -1$
\item or $f$ reverses orientation and $\trace(\textit{h}_{d-1}(f, X)) = 1,$
\end{itemize}
then $i(f, X) = 1$ and, in particular, $X$ contains a fixed point.
\end{corollary}
\begin{proof}
Theorem B shows that only one possibility may hold at the same time. By Szymczak's duality, the second and
third alternatives are equivalent to $\trace(\textit{h}_{1}(f^{-1}, X)) = -1$, hence we only need to prove the
statement for the first hypothesis.

In the case $X$ is an attractor, the conclusion is well-known.
 Otherwise, in Subsection \ref{subsec:attractorrepeller} we will see that
$\textit{h}_0(f, X)$ is trivial and, by Theorem B, $\trace(\textit{h}_r(f, X)) = 0$ for all $r > 1$.
Equation (\ref{eq:lefschetzlike}) then leads to
$$i(f, X) = - \trace(\textit{h}_1(f, X)) = 1.$$
\end{proof}

The restriction expressed by Corollary \ref{cor:indexinequality} is reflected in the possible sequences of integers
that can be realized as the sequence of fixed point indices $\{i(f^n, p)\}_{n \ge 1}$ of a fixed point $p$, isolated as an invariant set, of
a local orientation-reversing homeomorphism of $\mathds{R}^3$. Any such sequence must satisfy
some general relations, called Dold's congruences, see \cite{dold},
which will be briefly explained in Subsection \ref{subsec:permendo},
where we will also introduce the normalized sequences $\sigma^k$.
As a hint, let us say that every sequence $I = \{I_n\}_{n \ge 1}$ can be expressed uniquely as a formal linear combination of the
normalized sequences, $I = \sum_{k \ge 0} a_k \sigma^k$, and it satisfies Dold's congruences if and only if
all the coefficients $a_k$ are integers. In dimension 3, it is known that the sequence $\{i(f^n, p)\}_{n \ge 1}$ must be periodic,
no matter whether the orientation is reversed or not, see \cite{periodicindices}. We include an
alternative proof of this result which follows from Theorem A and Szymczak's duality.

\begin{corollary}
Let $f$ be a local homeomorphism of $\mathds{R}^3$ and $X$ an isolated invariant acyclic continuum. Then,
the sequence $\{i(f^n, X)\}_{n \ge 1}$ is periodic.
\end{corollary}
\begin{proof}
If $X$ is an attractor, the sequence $\{i(f^n, X)\}_{n \ge 1}$, is constant equal to 1; if
$X$ is a repeller and $f$ preserves orientation, it is constant equal to $-1$; if $X$ is a repeller and $f$
reverses orientation, one has $i(f^n, X) = (-1)^{n+1}$. Assume now that $X$ is neither an attractor, nor
a repeller and that $f$ preserves orientation. Using Szymczak's duality
 we obtain $\trace(\textit{h}_2(f^n, X)) = \trace(\textit{h}_1(f^{-n}, X))$.
By Theorem A, there exist two maps $\varphi : J \to J$ and $\varphi': J' \to J'$, with $J$ and $J'$ being finite sets,
such that $\trace(\textit{h}_1(f^n, X)) = -1 + \#\Fix(\varphi^n)$ and $\trace(\textit{h}_1(f^{-n}, X) = -1 + \#\Fix((\varphi')^n)$.
Note that the sequences $\{\#\Fix(\varphi^n)\}_{n \ge 1}$ and $\{\#\Fix((\varphi')^n)\}_{n \ge 1}$ are periodic.
Substituting in equation (\ref{eq:lefschetzlikeiterates}) we get that
$$i(f^n, X) = - \trace(\textit{h}_1(f^n, X)) + \trace(\textit{h}_2(f^n, X)) = - \#\Fix(\varphi^n) + \#\Fix((\varphi')^n),$$
and the conclusion follows.
The same formula holds for even iterates in the orientation-reversing case, whereas if $n$ is odd one has
$$i(f^n, X) = 2 - \#\Fix(\varphi^n) - \#\Fix((\varphi')^n).$$
This computation is also straightforward.
\end{proof}

The additional restrictions expressed by Corollary \ref{cor:indexinequality} and, more subtlety, by Theorems A and B on the
fixed point index sequence $\{i(f^n, p)\}_{n \ge 1}$ is the content of the following theorem.

\begin{theoremc}\label{thm:characterizationsequences}
Given a sequence $I = \{I_n\}_{n \ge 1} = \sum_{k \ge 1}a_k \sigma^k$, there exists an orientation-reversing local homeomorphism $f$ of
$\mathds{R}^3$ with a fixed point $p$
isolated as an invariant set and such that $I = \{i(f^n,p)\}_{n \ge 1}$ if and only if
\begin{itemize}
\item the coefficients $a_k$ are integers,
\item there are finitely many non-zero $a_k$,
\item $a_1 \le 1$ and $a_{k} \le 0$ for all odd $k > 1$.
\end{itemize}
\end{theoremc}

In Section \ref{sec:conleyindex} we will prove the necessity condition, which holds for any
 isolated invariant acyclic continuum $X$. However, the example constructed in Section
\ref{sec:sequences} to prove sufficiency can not be easily extended to acyclic continua
 which do not admit any neighborhood $U$ such that $U \setminus X$ is homeomorphic to $S^2 \times \mathds{R}$.

Corollary \ref{cor:indexinequality} has a nice dynamical application to minimal homeomorphisms.
 A map is minimal if there are no proper invariant sets.
 The question of existence of minimal homeomorphisms was raised by Ulam and is one
 of the problems contained in the ``Scottish Book'', see \cite{scottish}.
  It is known that every manifold over which $S^1$
 acts freely, for example any odd-dimensional sphere, admits minimal homeomorphisms,
  due to a work of Fathi and Herman, see \cite{fathiherman}. On the contrary,
the description of the sequences of fixed point indices of the iterates of a map led Le Calvez and Yoccoz
to proof the non-existence of minimal homeomorphisms in the finitely-punctured
2-sphere, see \cite{lecalvezyoccoz}.
Despite Theorem C does not apparently provide
enough insight in the question for orientation-reversing homeomorphism in the
finitely-punctured 3-sphere, Corollary \ref{cor:indexinequality} allows us to address this question for the case
of orientation-reversing homeomorphisms in $\mathds{R}^3$.

\begin{corollary}
If $f$ is a fixed point free orientation-reversing homeomorphism of $\mathds{R}^3$, then for every compact set $K\subset\mathds{R}^3$
there exists an orbit of $f$ disjoint from $K$. In particular, there are no minimal orientation-reversing homeomorphisms of $\mathds{R}^3$.
\end{corollary}
\begin{proof}
Denote $\bar{f}$ the extension of $f$ to $S^3$, leaving the point at $\infty$ fixed.
Since $\infty$ is the unique fixed point of $\bar{f}$, we deduce from Lefschetz-Dold Theorem that $i(\bar{f}, \infty) = 2$,
the Lefschetz number of $\bar{f}$.
Corollary \ref{cor:indexinequality} then implies that $\infty$ is not isolated as an invariant set and the conclusion follows.
\end{proof}

The remainder of the paper is organized as follows. Some preliminary algebraic and topological concepts
are introduced in Section \ref{sec:preliminaries}. In the next section, discrete Conley index is defined and
a new approach to the index and its duality is included.
The last subsection of Section \ref{sec:conleyindex} is devoted to
state Theorem \ref{thm:key}, which allows to prove Theorems A, B and the first part of Theorem C.
The proof of our results for a particular radial case is the content of Section \ref{sec:sequences},
where the second half of Theorem C is proved.
Section \ref{sec:proofthm} includes the proof of Theorem \ref{thm:key} and
the final section is devoted to give a brief overview of some results about fixed point index
of homeomorphisms.

\section{Preliminaries}\label{sec:preliminaries}

\subsection{Shift equivalence}\label{subsec:shiftequivalence}

In this article, we present the approach taken by Franks and Richeson in \cite{franksricheson} to
introduce the discrete Conley index. The key concept to understand their work is shift equivalence.

\begin{definition}
Let $\mathcal{K}$ be a category and $Y, Y'$ two objects of $\mathcal{K}$. Consider the endomorphisms
$g : Y \to Y$ and $g' : Y' \to Y'$. Then, $g$ and $g'$ are said to be \emph{shift equivalent} provided that
there exist morphisms $a : Y \to Y'$ and $b : Y' \to Y$ and an integer $m \ge 0$ such that
\begin{itemize}
\item $a \circ g = g' \circ a$, $b \circ g' = g \circ b$,
\item $b \circ a = g^m$ and $a \circ b = (g')^m$.
\end{itemize}
\end{definition}

It is not difficult to check that the notion of shift equivalence is an equivalence relation. The idea
behind this definition is to put in the same class all endomorphisms whose ``cores'' are conjugate,
where ``core'' means the largest automorphism obtained as a restriction of an endomorphism.
As a motivation, assume $g$ and $g'$ are bijective and shift equivalent. Then, if we define $h = a \circ g^{-m} =
(g')^m \circ a$ we obtain that $h^{-1} = b$ and $h$ is a conjugation between $g$ and $g'$, as it
satisfies $h^{-1} \circ g' \circ h = g$. Conversely, if $h$ is a conjugation then $a = h$ and $b = h^{-1}$
provide a shift equivalence between $g$ and $g'$.
The previous intuition will be formalized once we move onto the following particular cases.

Firstly, consider a very simple setting, the category formed by finite sets and maps between them.
\begin{definition}
Given a finite set $J$ and a map $\varphi : J \to J$, denote $\gim(\varphi)$ the largest invariant subset of $J$ under $\varphi$.
Then, we define the \emph{permutation induced by $\varphi$} as the
restriction of $\varphi$ to the subset $\gim(\varphi)$, and denote it by $\mathcal{L}(\varphi)$.
\end{definition}
Shift equivalence here only makes differences within the induced permutations, as the following proposition shows.
\begin{proposition}\label{prop:lerayreductionfinite}
Two finite maps $\varphi : J \to J$ and $\varphi' : J' \to J'$ are shift equivalent if and only if their
induced permutations are conjugate.
\end{proposition}
\begin{proof}
After the definition of shift equivalence we proved that this notion is equivalent to conjugation provided
that the endomorphisms are bijective. Therefore, we just need to show that a finite map is shift equivalent
to its induced permutation.

The set $J$ being finite, there exists $n_0$ such that $\varphi^{n_0}(J) = \gim(\varphi)$. Define $a = \varphi^{n_0} :
J \to \gim(\varphi)$ and let $b : \gim(\varphi) \to J$ be the inclusion map. It is easy to check that $a$ and $b$
provide a shift equivalence between $\varphi$ and $\mathcal{L}(\varphi)$.
\end{proof}

The category of finite-dimensional vector spaces and linear endomorphisms has a richer structure in which
shift equivalence will be completely described as well. As shown previously, for linear automorphisms it is equivalent
to conjugation.
Before we prove the characterization, we need to introduce several definitions concerning linear endomorphisms.
We stick our considerations to vector spaces over $\mathds{Q}$.
Let $H$ be a finite-dimensional vector space and $u : H \rightarrow H$ an endomorphism. The \emph{generalized
kernel} of $u$, $\gker(u)$, is the set of vectors which are eventually mapped onto 0, that means
$$\gker(u) = \bigcup_{n \ge 0} \ker(u^n).$$
The \emph{generalized image} of $u$, $\gim(u)$, is the largest invariant subspace of $H$ under $u$,
$$\gim(u) = \bigcap_{n \ge 0} \im(u^n).$$
It is well known (and easy to check) that there exists an integer $n_0$ such that
$$ \ker(u)\subsetneq \ker(u^2)\subsetneq\dots\subsetneq \ker(u^{n_0})= \ker(u^{n_0+1})$$and
$$ \im(u^{n_0+1})=\im(u^{n_0})\subsetneq \dots \subsetneq \im(u^{2})\subsetneq \im (u)$$ and that one has $$\gker(u) =\ker (u^n), \enskip \gim(u)=\gim(u^n)$$ for every $n\geq n_0$.
The subspaces $\gker(u)$ and $\gim(u)$ are obviously positively invariant.
The restriction of $u$ to the subspace $\gim(u)$ is called in the literature \emph{Leray reduction} and will
be denoted $\mathcal{L}(u)$. It is easy to prove that
$$H = \gker(u) \oplus \gim(u).$$

This decomposition gives a synthetic description of $u$:
on one hand, the restriction of $u$ to $\gker(u)$ is nilpotent,
and on the other hand, the restriction of $u$ to $\gim(u)$, that is, its Leray reduction $\mathcal{L}(u)$ is an automorphism.
Observe that $u_{\vert \gker(u)}$ being nilpotent, its trace is 0 and
$$\trace(u) = \trace(\mathcal{L}(u)).$$
Observe also that for any $n \geq 1$ one has
$$\gker(u^n) = \gker(u), \enskip \gim(u^n) = \gim(u), \enskip \mathcal{L}(u^n) = \mathcal{L}(u)^n,$$ hence
$$\trace(u^n) = \trace(\mathcal{L}(u)^n).$$

This remark shows that the Leray reduction of a linear endomorphism determines its trace
and the trace of its iterates.
More importantly, Leray reduction characterizes the shift equivalence class of a linear endomorphism.

Note that for linear endomorphisms in finite-dimensional vector spaces we have that $\gim(u) = \im(u^{n_0})$ for some positive integer $n_0$,
as happened for finite maps. The argument presented in Proposition \ref{prop:lerayreductionfinite}
can also be applied to prove that a linear endomorphism is shift equivalent to its Leray reduction.
Since Leray reductions are automorphisms, hence bijective, the next proposition follows.

\begin{proposition}
Two endomorphisms $u : G \to G$, $v : H \to H$ are shift equivalent if and only if their Leray reductions are conjugate.
\end{proposition}

As a consequence, linear shift equivalent endomorphisms have equal traces. This trivial observation
together with the following proposition show how some trace computations will be done in this article.

\begin{lemma}\label{lem:gimgker}
Let $H$ be a finite-dimensional vector space and $u: H\to H$ an endomorphism. We suppose that:
\begin{itemize}
\item $F$ is a subspace that is invariant under $u$ and included in its generalized kernel.
\item $G$ is a subspace that is invariant under $u$ and that contains both $F$ and $\gim(u)$.
\end{itemize}
Then, the naturally induced endomorphism $\widehat u: G/F\to G/F$ is shift equivalent to $u$. In particular,
their traces are equal.
\end{lemma}
\begin{proof}
Since $F$ and $G$ are invariant under $u$, the map $\widehat{u}$ is well-defined and $\widehat{u}(x + F) = u(x) + F$
for any $x \in G$. It follows that, for any $n \ge 0$, $\widehat{u}^n(x + F) = u^n(x) + F$, hence
$\gim(\widehat{u}) = \gim(u) + F$.
The projection map $\pi: G \to G/F$ induces an isomorphism between $\gim(u)$ and $\gim(\widehat{u})$ because
$\gim(u) \cap F \subset \gim(u) \cap \gker(u) = \{0\}$. Since $\pi \circ u_{\vert G} = \widehat{h} \circ \pi$,
we deduce that $\pi_{\vert \gim(u)}$ is a conjugation between $\mathcal{L}(u)$ and $\mathcal{L}(\widehat{u})$.
\end{proof}

Clearly, any linear endomorphism and its Leray reduction have the same non-zero complex eigenvalues, counted with multiplicity.
Therefore, the spectra, except for the eigenvalue 0, of linear endomorphisms is invariant under shift equivalence.
However, two linear automorphisms having equal spectra are not always conjugate because their Jordan
canonical forms may differ. Thus, the following definition is slightly weaker than
shift equivalence.

\begin{definition}\label{def:spectrumequivalent}
Two endomorphisms $u : G \to G$ and $v : H \to H$ are \emph{spectrum equivalent} if they have
the same non-zero complex eigenvalues, counted with multiplicity, or, equivalently,
$$\trace(u^n) = \trace(v^n)$$
for every positive integer $n$.
\end{definition}

\subsection{Permutation endomorphisms}\label{subsec:permendo}

We now introduce a class of maps which will help us to describe the first homological discrete Conley index.
The definition has a truly combinatorial taste, as the idea behind it will be to describe combinatorially
the dynamics around an isolated invariant acyclic continuum.

\begin{definition}
A \emph{permutation endomorphism} $u : H \to H$ is an endomorphism for which there exists a
map $\varphi : J \to J$ over a finite set $J$, where $\# J = dim(H)$,
 and a basis $\{e_j\}_{j \in J}$ of $H$ such that
$u(e_j) = e_{\varphi(j)}$. If the map $\varphi$ is bijective we say $u$ is a \emph{permutation automorphism}.
\end{definition}

Once we move onto higher homological indices the notion of permutation endomorphism is too rigid, but a generalization
of it will serve us to describe them.

\begin{definition}\label{def:dominatedendo}
An endomorphism $u : H \to H$ is \emph{dominated} by a finite map $\varphi : J \to J$ if there exists a decomposition
$H = \bigoplus_{j \in J} H_j$ such that $u(H_j) \subset H_{\varphi(j)}$ for every $j \in J$.
\end{definition}

Observe that if $u$ is dominated by a fixed-point free finite map $\varphi$, then $\trace(u) = 0$. This trivial
remark indicates the way we will prove Theorem B.

In the computations, reduced homology groups will naturally appear. As a consequence, permutation endomorphisms
do not exactly provide the required description of the first homological index and we need to include the
following definition.

\begin{definition}
A {\em reduced permutation endomorphism (automorphism)} is obtained from a permutation endomorphism (automorphism) $u: H\to H$
associated to a basis $\{e_j\}_{j\in J}$ and a map $\varphi: J\to J$ by taking the restriction $v$
of $u$ to $\ker(\delta)$ where $\delta :H\to \mathds{Q}$ is the linear form that sends each $v_j$ to $1$.
\end{definition}
As $u$ and $v$ are completely determined by $\varphi$, up to conjugacy, we will sometimes say that $\varphi$ defines
the permutation endomorphism $u$ and the reduced permutation endomorphism $v$.
\begin{proposition}\label{prop:reducedpermutationtrace}
The reduction $v$ of a permutation endomorphism $u$ satisfies $$\trace(v^n) = \trace(u^n) - 1,$$ for every $n \ge 1$.
\end{proposition}
\begin{proof}
Consider a basis of $\ker(\delta)$ and a vector $e_j$ which extends it to a basis of $H$. Then, for every $n \ge 1$,
$\delta(u^n(e_j)) = 1$, hence $u^n(e_j) - e_j \in \ker(\delta)$ and the formula for the traces follows.
\end{proof}

The final part of this section is devoted to show how the traces of permutation endomorphisms
and its iterates can be computed. We will deduce that this sequence of traces satisfies the so-called Dold's congruences, which were first introduced in \cite{dold}.

\begin{definition}
A sequence of integers $I = \{I_n\}_{n \ge 1}$ is said to satisfy \emph{Dold's congruences} if, for every $n \ge 1$,
$$\sum_{k|n} \mu(n/k)I_k \equiv 0 \;(\mathrm{mod} \; n).$$
\end{definition}
The Möbius function $\mu$ assigns to each natural number $n$ a value $-1, 0$ or $1$, depending
on its prime decomposition. If a factor appears at least twice in the prime decomposition then $\mu(n) = 0$,
otherwise $\mu(n) = (-1)^s$, where $s$ is the number of prime factors of $n$.
 Dold's congruences can be described in a more elementary way.
Consider the following \emph{normalized sequences} $\sigma^k = \{\sigma^k_n\}_{n \ge 1}$, where
$$
\sigma^k_n =
\begin{cases}
k & \text{if } n \in k \mathds{N} \\
0 & \text{otherwise.}
\end{cases}
$$
Every sequence $I = \{I_n\}_{n \ge 1}$ can be written as a formal combination of
normalized sequences, $I = \sum_{k \ge 1} a_k \sigma^k$. A sequence $I$ satisfies Dold's congruences
if and only if all coefficients $a_k$ are integers. As observed by Babenko and Bogatyi, see \cite{babenko},
 a sequence satisfying Dold's congruences is periodic if and only if it can be written
as a finite combination of normalized sequences.

Let us conclude this section with the following result:

\begin{proposition}\label{prop:permutationendotrace}
Let $u$ be a permutation endomorphism defined by a finite map $\varphi$.
Then, the sequence $\{\trace(u^n)\}_{n \ge 1}$ is equal
to the sequence $\{\#\Fix(\varphi^n)\}_{n \ge 1}$, where $\#\Fix(\varphi^n)$ denotes the number of fixed points of $\varphi^n$. In addition, it satisfies Dold's congruences.
\end{proposition}
\begin{proof}
The first statement of the proposition is obvious and we only need to prove the last statement. If $a_k$ denotes the number of $k$-periodic orbits
of the map $\varphi$ it is clear that
$$\#\Fix(\varphi^n)= \sum_{k | n} k \cdot a_k$$ and
it follows that $$\{\#\Fix(\varphi^n)\}_{n \ge 1} = \sum_{k \ge 1} a_k \sigma^k.$$
\end{proof}

\subsection{Acyclic continua}\label{subsec:acyclic}

The class of compact subsets of $\mathds{R}^d$ to which our results apply is that of acyclic continua.
In the literature, the term acyclic is often used to call sets such that
all its reduced homology groups vanish. Since we deal with compact sets,
the more appropriate homology theory for our setting is \v{C}ech homology and we will say that a
 set $K$ is \emph{acyclic} if $\check{H}_r(K) = 0$ for all $r \ge 1$. All the homology groups appearing
 in this article are equipped with rational coefficients.

The following proposition gives an alternative definition of acyclic continuum, which is the one we will use later.

\begin{proposition}\label{prop:acyclic}
A continuum $K \subset \mathds{R}^d$ is acyclic if and only if
for every neighborhood $U$ of $K$, there exists another neighborhood $V \subset U$ of $K$ such that
the inclusion-induced maps $H_r(V) \to H_r(U)$ are trivial for every $r \ge 1$.
\end{proposition}
\begin{proof}
Let us prove sufficiency first. Inductively, construct a basis of neighborhoods $\{U_n\}_{n \ge 1}$ of $K$
such that the inclusion-induced maps $H_r(U_{n+1}) \to H_r(U_n)$ are trivial for all $n, r \ge 1$.
The continuity property of \v{C}ech homology implies that $\check{H}_r(K)$ is isomorphic to
the inverse limit of the inverse system formed by the groups $\{H_r(U_n)\}_{n \ge 1}$ together
with the linear maps $p_{m,n} : H_r(U_m) \to H_r(U_n)$, for $m \ge n$, induced by inclusion.
Since all these maps are trivial by assumption, the inverse limit is zero and all
\v{C}ech homology groups are trivial for $r \ge 1$, hence $K$ is acyclic.

Conversely, take a basis of neighborhoods $\{U_n\}_{n \ge 1}$ of $K$ composed of compact polyhedra.
Since $K$ is acyclic, $\varprojlim H_r(U_n) = 0$.
Given any $n \ge 1$, consider the nested sequence $\{\im(p_{m,n})\}_{m \ge n}$ of vector subspaces of $H_r(U_n)$
 and denote their intersection $A_n$. Since $H_r(U_n)$ is finite dimensional, there exists $n_0 \ge n$ such that
$\im(p_{n_0, n}) = A_n$.
For every $k \ge m \ge n$ we have that $p_{k,n} = p_{k, m} \circ p_{m,n}$, hence the subspaces $A_n$
are mapped onto each other, that is, $p_{m,n}(A_m) = A_n$ for every $m \ge n$.
Thus, for a given $x \in A_n$ it is possible to construct a sequence $\{x_m\}_{m \ge n}$
such that $x_n = x$ and $p_{m+1, m}(x_{m+1}) = x_m$ for any $m \ge n$. As a consequence, if any $A_n \neq \{0\}$ the inverse limit of
the sequence $\{H_r(U_n)\}_{n \ge 1}$ can not be trivial. Therefore, for any $n$ one can find $n_0$ so that
$\im(p_{n_0,n}) = \{0\}$, that is, the map $p_{n_0,n} : H_r(U_{n_0}) \to H_r(U_n)$ is the zero map.
After taking the maximum of a finite set of bounds, we can assume the latter holds for every $r \ge 1$.
For every element $U$ of $\{U_n\}_{n \ge 1}$ we have proved that there is another element of the sequence, $V$,
such that $H_r(V) \to H_r(U)$ is trivial for every $r \ge 0$. Since $\{U_n\}_{n \ge 1}$ is a basis of neighborhoods
of $K$, the proof is finished.
\end{proof}

The previous proposition shows that the isomorphism between $\check{H}_r(K)$ and the trivial group extends
to an isomorphism between the pro-group formed by the groups $H_r(U_n)$ and the inclusion-induced maps and the
zero pro-group, for any basis of open neighborhoods $\{U_n\}_{n \ge 1}$ of $K$. This ultimately comes from the fact that, since
we use rational coefficients, the sequence $\{H_r(U_n)\}_{n \ge 1}$ satisfies the Mittag-Leffler condition.
This result was first proved, in a more general setting, in \cite{Keesling} (see also \cite{MardesicSegal}).

\subsection{Isolating blocks and filtration pairs}

Our aim in this subsection is to introduce some topological notions that
are necessary to describe discrete Conley index theory.
Throughout the article we will often work with the relative topology of closed
subsets of $\mathds{R}^d$ or, more generally, manifolds with boundary $N$.
Given two subsets $A \subset B$ of $N$ such that $B$ is closed,
 $\overline{A}$ will denote the closure of $A$,
$\int(A)$ and $\int_B(A)$ the interior of $A$ and its interior relative to $B$ and
$\partial A = \overline{A} \setminus \int(A)$ and $\partial_B A = \overline{A} \cap (\overline{B \setminus A})
= \overline{A} \setminus \int_B(A)$,
 the boundary and relative boundary of $A$, respectively.

Let $M$ be a $d$-manifold without boundary. A \emph{regular decomposition} $(M_1, M_2)$ of $M$
is a pair of $d$-submanifolds with boundary such that $M_1 \cup M_2 = M$ and $\partial M_1 = \partial M_2$.
Similarly, a regular 2-decomposition $(N_1, N_2)$ of a $d$-manifold with boundary $N$
is a pair of $d$-manifolds with boundary such that there exist regular
decompositions $(M_{1,1}, M_{1,2})$ of $\partial N_1$ and $(M_{2,1}, M_{2,2})$ of $\partial N_2$ satisfying
$$N_1 \cup N_2 = N, \;\;\; N_1 \cap N_2 = M_{1,2} = M_{2,1}.$$
In that case $\partial N = M_{1,1} \cup M_{2,2}$ and every connected component of $N_1 \cap N_2$
defines an element of the group $H_{d-1}(N, \partial N)$.
We will need also the notion of regular 3-decomposition $(N_1, N_2, N_3)$ of $N$.
It consists of a triple of $d$-manifolds with boundary such that there exists a regular decomposition $(M_{1,1}, M_{1,2})$ of $\partial N_1$,
a regular decomposition $(M_{3,2}, M_{3,3})$ of $\partial N_3$, a regular decomposition $(M_{1,3}, M_{2,2})$ of $\partial N_2$ and a partition $M_{1,3}=M_{1,2}\sqcup M_{2,3}$ in open and closed submanifolds such that
$$N_1\cup N_2\cup N_3 = N, \quad N_1\cap N_2=M_{1,2}=M_{2,1}, \quad N_2\cap N_3=M_{2,3}=M_{3,2}, \quad N_1\cap N_3=\emptyset.$$
\noindent
In that case $\partial N= M_{1,1}\cup M_{2,2}\cup M_{3,3}$ and every connected component of $M_{1,3}$ defines an element of $H_{d-1}(N,\partial N)$.

A pair $(N, L)$ of subsets of $\mathds{R}^d$ is said to be \emph{regular} if $N$ is a $d$-manifold with bicollared boundary and
$(\overline{N \setminus L}, L)$ is a regular 2-decomposition of $N$.
Similarly, a triple $(N, L, L')$ is \emph{regular} provided that $(L', \overline{L \setminus L'}, \overline{N \setminus L})$
is a regular 3-decomposition of $N$.
Note that for a regular pair $(N, L)$, the quotient space $N/L$ is guaranteed to be an ANR.
Consequently, there is a natural isomorphism $\widetilde{H}_*(N/L) \sim H_*(N, L)$ between
reduced and relative homology groups.

Next, we define some dynamical concepts.
Consider a local homeomorphism $f$ of $\mathds{R}^d$.
Given a compact set $N$ included in its domain, the \emph{stable set} of $f$ in $N$, $\Lambda^+(f, N)$,
consists of the set of points of $N$ whose forward orbit remains in $N$,
$$\Lambda^+(f, N) = \bigcap_{n \ge 0}f^{-n}(N).$$
Similarly, the \emph{unstable set} $\Lambda^-(f, N)$ of $f$ in $N$ consists of the set of points of $N$ whose backward orbit
remains in $N$,
$$\Lambda^-(f, N) = \bigcap_{n \ge 0}f^{n}(N).$$
Of course, the largest subset of $N$ invariant under $f$ is exactly $\Lambda^+(f, N) \cap \Lambda^-(f, N)$.
Additionally, if $N$ is compact so are the stable and unstable sets.

Denote $\Inv(N)$ the maximal invariant subset of a compact set $N \subset \mathds{R}^d$.
In the case $\Inv(N) \subset \int(N)$, we say that $N$ is an \emph{isolating neighborhood}.
A compact invariant set $X$ for which there exists an isolating neighborhood $N$ such that $\Inv(N) = X$ is called
an isolated invariant set.
Note that for any isolating neighborhood $N$ of $X$ we have $X = \Lambda^+(f,N) \cap \Lambda^-(f,N)$.
Recall that if $X$ has a compact neighborhood $N$ such that
$\bigcap_{n \in \mathds{N}} f^n(N) = X$ or $\bigcap_{n \in \mathds{N}} f^{-n}(N) = X$ then $X$
is called an attractor or a repeller, respectively.
Finally, a compact set $N$ which presents no interior ``discrete tangencies'', i.e.
$$f^{-1}(N) \cap N \cap f(N) \subset \int(N)$$
is called an \emph{isolating block}.
Notice that every isolating block is also an isolating neighborhood.
It is known that every isolated invariant set admits a fundamental system of neighborhoods which are isolating blocks,
see for example \cite{easton}.
From the definition, we see that if $N'$ is close enough to $N$ then $N'$ is also an isolating block.
Thus, isolating blocks can be chosen to be $d$-manifolds with bicollared boundary.
Now, we introduce a concept which is due to Franks and Richeson, see \cite{franksricheson}.

\begin{definition}
A pair of compact sets $(N, L)$ is a \emph{filtration pair} for the invariant set $X$
provided $N$ and $L$ are the closure of their interiors and the following
properties are satisfied:
\begin{itemize}
\item $\overline{N \setminus L}$ is an isolating neighborhood of $X = \Inv(\overline{N \setminus L})$.
\item $L$ is a neighborhood of the exit set of $N$, $N^- = \{x \in N: \; f(x) \notin \int(N)\}$.
\item $f(L) \cap (\overline{N \setminus L}) = \emptyset$.
\end{itemize}
\end{definition}

This is the object which will be most used throughout the article. The following proposition,
based in the robustness of the definition of filtration pair,
shows that it can always be assumed to satisfy nice local properties.

\begin{proposition}
Regular filtration pairs $(N, L)$ exist for any isolated invariant set $X$.
\end{proposition}
\begin{proof}
Let $N$ be an isolating block for $X$ which is also a $d$-manifold with bicollared boundary.
Define $N^- = \{x \in N : f(x) \notin \int(N)\}$.
Clearly, $f(N^-) \cap N \subset \partial N$
and that implies, since $N$ is an isolating block, that for every $x \in f(N^-) \cap N$ its
image is $f(x) \notin N$.
Consequently, the compact sets $f(N^-)$ and $\overline{N \setminus N^-}$ are disjoint.

Therefore, if we consider a small neighborhood $L$ of $N^-$ in $N$,
we have that $f(L) \cap (\overline{N \setminus L}) = \emptyset$ and also that
$\overline{N \setminus L}$ is an isolating neighborhood of $X$.
Choosing $L$ to fit in a regular pair $(N, L)$, we obtain the desired regular filtration pair for $X$.
\end{proof}

Additionally, this proof shows that it is possible to find regular filtration pairs as close to $X$ as required.

\section{Discrete Conley index theory}\label{sec:conleyindex}

\subsection{Conley index and Lefschetz-Dold Theorem}\label{subsec:indexcomputation}

Let $f$ be a local homeomorphism of $\mathds{R}^d$ and $X$ an isolated invariant set.
Consider a regular filtration pair $(N, L)$ for $X$.
Denote $\pi_L: N\to N/L$ the projection onto the quotient space $N/L$ that sends every point
 $z\in N\setminus L$ onto itself and every point $z\in L$ onto the point $[L]$.
The definition of filtration pair permits us to define a continuous map $\bar{f} : N/L\to N/L$ that fixes $[L]$
and sends every point $z\in N\setminus L$ onto $\pi_L(f(z))$. This induced map, which
appears to depend strongly in the filtration pair chosen turns out to ultimately depend only
on the invariant set $X$, up to shift equivalence.
This notion was defined in Subsection \ref{subsec:shiftequivalence} and
applies to the category $\textbf{Top}_*$, whose objects are pointed topological spaces
and the morphisms are continuous base-preserving maps.
The importance of shift equivalence is highlighted by the following theorem.

\begin{theorem}
All maps $\bar{f}$ arising from filtration pairs $(N, L)$ of $X$ are shift equivalent.
\end{theorem}

The \emph{discrete Conley index of $X$} is defined as the shift equivalence class of the map $\bar{f}$.
For a complete proof of this theorem we refer the reader to \cite{franksricheson}.

We could have also considered the induced maps $\bar{f}_{*,r} : \widetilde H_r(N/L) \rightarrow \widetilde H_r(N/L)$
in the reduced homology groups with rational coefficients.
Then, all possible endomorphisms $\bar{f}_{*,r}$ are shift equivalent, in other words,
their Leray reductions $\mathcal{L}(\bar{f}_{*,r})$ are conjugate.
The shift equivalence class of $\bar{f}_{*,r}$
is called \emph{$r$-homological discrete Conley index of $X$} and denoted $\textit{h}_r(f, X)$.

Lefschetz-Dold Theorem can be applied to the map $\bar{f} : N/L \to N/L$.
  From the definition of filtration pair we get that $\bar{f}$ is locally constant
at $[L]$, hence $i(\bar{f},[L]) = 1$, and $f$ and $\bar{f}$ are locally conjugate
 around $X = \Inv(\overline{N \setminus L})$. Lefschetz-Dold Theorem yields
\begin{equation*}
\Lambda(\bar{f}) = i(\bar{f}, N/L) = 1 + i(f, X),
\end{equation*}
where $\Lambda(\bar{f})$ denotes the \emph{Lefschetz number} of $\bar{f}$,
defined as the alternate sum of the traces of the maps induced by $\bar{f}$
in the singular homology groups $H_r(N/L)$, $r \ge 0$.
A similar equation can be deduced for any iterate of $f$. For $n \ge 1$,
\begin{equation}\label{eq:lefschetzthmiterates}
\Lambda((\bar{f})^n) = 1 + i(f^n, X).
\end{equation}

This last equation requires some clarification. Despite the pair $(N, L)$ is not a filtration pair, in general,
for $X$ and the map $f^n$, we can define a map $\overline{f^n} : N/L \to N/L$ by fixing the basepoint $[L]$
and sending $x \in N \setminus L$ to $[L]$ if any of its first $n$ forward images lies in $L$ and to $f^n(x)$ otherwise.
It is not difficult to see that $\overline{f^n} = (\bar{f})^n$, hence equation (\ref{eq:lefschetzthmiterates})
results from applying Lefschetz-Dold Theorem to the map $\overline{f^n}$.

 We will estimate the fixed point indices $i(f^n, X)$ by examining the Lefschetz numbers, $\Lambda((\bar{f})^n)$.
More concretely, we would like to compute the traces of the finite-dimensional linear maps
$$(\bar{f}_{*,r})^n: \widetilde H_r(N/L)\to \widetilde H_r(N/L),\enskip 0\leq r\leq d.$$
Expanding the definition of $\Lambda((\bar{f})^n)$ we obtain
\begin{equation}\label{eq:1}
\Lambda((\bar{f})^n) = 1+\sum_{r = 0}^d (-1)^r \trace((\bar{f}_{*,r})^n)
= 1+\sum_{r = 0}^d(-1)^r \trace(\textit{h}_r(f^n, X)),
\end{equation}
where the extra 1 makes up for the small gap between the reduced and singular homology groups at grade 0.
Note that all higher homology groups of $N/L$ are trivial because $N$ and $L$ are $d$-manifolds with boundary.
The notion of trace of $\textit{h}_r(f, X)$ is well-defined because
traces are invariant under shift equivalence.
Substituting equation (\ref{eq:1}) into (\ref{eq:lefschetzthmiterates}), we obtain equation (\ref{eq:lefschetzlikeiterates}),
$$i(f^n, X) = \sum_{r = 0}^d(-1)^r \trace(\textit{h}_r(f^n, X)).$$

\subsection{Attractors and repellers}\label{subsec:attractorrepeller}

The computation of the fixed point index $i(f^n, X)$ and the traces of the homological
 Conley indices are particularly easy when $X$ is an acyclic continuum and either
an attractor or a repeller. In such situations, the work done in Proposition 3 of \cite{periodicindices}
guarantees the existence of suitable filtration pairs to work with. If $X$ is an attractor then
we can assume that $N$ is connected and $L$ is empty, hence
$$\trace(\textit{h}_0(f^n, X)) = 1,$$
and a nilpotence argument, which is described in \cite{richesonwiseman} and
is a particular case of the one we present in Subsection \ref{subsec:nilpotence}, shows
that the maps $\bar{f}_{*,r}$ are nilpotent, hence the traces of $\textit{h}_r(f^n, X)$ are 0 for $r \ge 1$.
Substituting in equation (\ref{eq:lefschetzlikeiterates}) we get that $i(f^n, X) = 1$, as expected.

In the case $X$ is a repeller, $N$ is connected, $L$ is not empty and we can assume that there is a unique
connected component of $\mathds{R}^d \setminus L$ contained in $N$.
 Therefore, $H_d(N/L) \sim \mathds{Q}$ and
$$\trace(\textit{h}_d(f^n, X)) = d(f)^n.$$
By connectedness of $N/L$, we obtain that $\trace(\textit{h}_0(f^n, X)) = 0$.
This result agrees with Szymczak's duality, which will be described in Subsection \ref{subsec:duality}.
Applying this duality we obtain
$\trace(\textit{h}_r(f^n, X)) = \trace(\textit{h}_{d-r}(f^{-n}, X))$ and, since a repeller for $f$ is an attractor for $f^{-1}$,
the previous computations show that $\trace(\textit{h}_r(f^n, X)) = 0$ for any $0 < r < d$.
Finally, we obtain that $i(f^n, X) = (-1)^{d} d(f)^n$, which agrees with the well-known formula for
repelling fixed points.

These computations prove Theorem A for attractors and repellers.
Simply set, for repellers in $d > 1$ and attractors, $\varphi$ to be the identity
map in a set consisting of one or two points, respectively. In the remaining case, a repeller in dimension $d = 1$,
define $\varphi$ as the identity in a two-element set if $f$ preserves orientation and as the permutation that swaps the elements
if $f$ reverses orientation.
The hypothesis of Theorem B are never satisfied for these particular invariant sets.
Theorem C also follows easily, as we have shown that the sequence $\{i(f^n,p)\}_{n \ge 1}$ is equal to
$\sigma^1$ for attractors and to $\sigma^1 - \sigma^2$ for repellers when $f$ is a local orientation-reversing
homeomorphism of $\mathds{R}^3$.

In the case $X$ is neither an attractor nor a repeller, we can assume also that
$N$ is connected, $L$ is not empty and $H_d(N/L)$ is trivial.
This implies that $\trace(\textit{h}_0(f^n, X))  = 0$
 and $\trace(\textit{h}_d(f^n, X)) = 0$. However, no further
assumptions on the filtration pair $(N, L)$ can be made. Equation (\ref{eq:lefschetzlikeiterates})
simplifies into
\begin{equation}\label{eq:lefschetzsimple}
i(f^n, X) = \sum_{r = 1}^{d-1} \trace(\textit{h}_r(f^n, X)).
\end{equation}
In particular, for local homeomorphisms of $\mathds{R}^3$ we will need just to examine
the maps $\bar{f}_{*,1}$ and $\bar{f}_{*,2}$ in order to compute $i(f^n, X)$.

\subsection{Another approach to the discrete Conley index and its duality}\label{subsec:duality}

In this subsection, we present another possible approach to discrete Conley index which makes a closer geometrical
connection with the continuous-time case. Charles Conley originally introduced the index as
a topological invariant valid for flows, see \cite{conley}.
One can define the notions of isolated compact invariant set $X$ and isolating block $N$ of a flow in a similar way it was defined in the discrete case.
The exit set $l^-$ of a given isolating block $N$ for $X$ is a closed subset of the boundary of $N$.
The (continuous) Conley index of $X$ is the homotopy class of the pointed
space $(N/l^-, [l^-])$, which does not depend on the choice of $N$ and $l^-$ but only
in the local dynamics of the flow around $X$. Usually, the Conley index is defined using
index pairs, which are more general than the pairs $(N, l^-)$ we introduced as, for instance, they do not
require the second set to be contained in the boundary of the first one.

The discrete Conley index has been introduced in this article using filtration pairs $(N, L)$, a notion
slightly less general than that of index pairs, for the discrete case, which is widespread in the literature.
However, in any case, the exit set of the isolating neighborhood $N$ is contained in $L$
and, consequently, $L$ is not contained in $\partial N$ except for highly degenerate cases.
We present below a way of computing discrete Conley index using pairs $(N, l)$
satisfying $l \subset \partial N$. Additionally, the use of isolating blocks will allow us to
compute the indices of a map and its inverse at once.

Let $X$ be an isolated invariant set of a local homeomorphism $f$. Take an isolating block $N$ for $X$ which is a manifold with bicollared boundary. The two sets
$$U^-=\{x\in N\, \vert \, f(x)\not\in N\}, \;\enskip U^+=\{x\in N\, \vert \, f^{-1}(x)\not\in N\}$$
are open sets of $N$ whose union covers $\partial N$.
Thus, there exists a regular decomposition $(l^-, l^+)$ of $\partial N$ such that $l^-\subset U^-$ and $l^+\subset U^+$.

As we did with filtration pairs, it is possible to define a map $f^- : N/l^- \to N/l^-$
which fixes $[l^-]$ and sends any $x \in N \setminus l^-$ to either $f(x)$ if $f(x) \in N \setminus l^-$
or the basepoint $[l^-]$ in other case. Similarly, we can define a map $f^+ : N/l^+ \to N/l^+$ using
$f^{-1}$ instead of $f$.

\begin{lemma}\label{lem:f-f+continuous}
The maps $f^-$ and $f^+$ are continuous.
\end{lemma}
\begin{proof}
We only sketch a proof for $f^-$, the other case being completely analogous.
Observe that if $x \in N$ and $f(x) \in \partial N$ then we must have that $f(x) \in l^-$, because
$f(x)\not\in U^+$. This yields that $(f^-)^{-1}([l^-])$ is a
compact neighborhood of $[l^-]$ and continuity follows easily.
\end{proof}

A triple $(N, l^-, l^+)$ as defined will be called \emph{filtration triple}. Our task is now
to show that the discrete Conley index can be extracted from this framework, and that the symmetry between $f$ and $f^{-1}$ permit us to give a proof of a duality theorem originally due to Szymczak (see \cite{szymczak}). Consider a compact set $L$ such that
$(N, L)$ is a filtration pair, and let $\bar{f}$ be the induced map in $N/L$.

\begin{proposition}
The maps $\bar{f}$ and $f^-$ are shift equivalent.
\end{proposition}
\begin{proof}
There exists an integer $n \ge 1$ so that, for every $x \in L$, one of the first $n$ forward images of $x$ by $f$
does not lie in $N$.
The complementary $U$ of $\bigcap_{k = 0}^n f^{-k}(N)$ in $N$  is composed of the points $x \in N$ such that
$f^k(x) \notin N$ for some $1 \le k \le n$ and so contains $L$. Denote
$\pi_{l^-}$ and $\pi_L$ the projections of $N$ onto $N/l^-$ and $N/L$ respectively.
The map, from $N$ to $N/l^-$, that sends every point of $U$ to $[l^-]$ and every point of $x\in\bigcap_{k = 0}^n f^{-k}(N)$ to $\pi_{l^-}(f^n(x))$ induces a  map $b: N/L\to N/l^-$ because $L\subset U$. The map $b$ is constant on the open set $U$ of $N$ and continuous when restricted to $\bigcap_{k = 0}^n f^{-k}(N)$. Like in Lemma \ref{lem:f-f+continuous}, to check the continuity of $b$ on $N/L$ we just need to observe
that $f^n(x) \in l^-$ for any $x \in \partial_N U$, hence $b(x) = [l^-]$ as desired.
To prove that $b \circ \bar{f}=f^- \circ b$, one must prove that
$$b(\bar f(\pi_L(x)))= f^-(b(\pi_L(x)))$$
for every $x\in N$. Observe that this equation is satisfied in the cases
$$x\in L, \enskip x\in U\setminus L, \enskip x\in \bigcap_{k = 0}^{n+1} f^{-k}(N), \enskip x\in \bigcap_{k = 0}^n f^{-k}(N)\setminus \bigcap_{k = 0}^{n+1} f^{-k}(N).$$
The projection $\pi_L$ sends every point of $l^-$ to $[L]$ because $l^- \subset L$ and so induces a continuous map $a : N/l^- \rightarrow N/L$ which sends
$[l^-]$ to $[L]$. To prove that $a \circ f^-=\bar{f}\circ a$, one must prove that
$$a(f^-(\pi_{l^-}(x)))= \bar f(a(\pi_{l_-}(x)))$$ for every $x\in N$. Here again, using the fact that $U^-\subset L$ and $f(L) \cap N \subset L$, one can observe that the previous equality is satisfied in the following cases
$$x\in l^-, \enskip x\in U^-\setminus l^-,\enskip x\in L\setminus U^-,\enskip x\in N\setminus L.$$
It remains to prove that $b \circ a = (f^-)^n$ and $a \circ b = (\bar{f})^n$, which means
$$b(a(\pi_{l^-}(x)))=(f^-)^n(\pi_{l^-}(x)), \;\enskip a(b(\pi_L(x)))=(\bar{f})^n(\pi_L(x))$$ for every $x\in N$. It is satisfied in the following cases
$$x\in U,\enskip x\in\partial_N U, \enskip x\in N\setminus \overline U.$$
\end{proof}

There is a natural isomorphism $\widetilde H_r(N/l^-)\sim H_r(N,l^-)$ and so the map $f^-$ induces an endomorphism $f^-_{*,r} : H_r(N, l^-) \to H_r(N, l^-)$.
The previous proposition shows that $f^-_{*,r}$ is a representative of the shift equivalence class $\textit{h}_r(f, X)$.
Let us explain more precisely the meaning of $f^-_{*,r}$.
Consider a one-to-one continuous map
$$H:[0,1)\times \partial N\to \mathds{R}^d\setminus \int(N),$$
 such that $H(0,z)=z$ for every $z\in \partial N$ and such that $N\cup H([0,1)\times \partial N)$
is a neighborhood of $N$ and denote $\rho: N\cup H([0,1)\times \partial N)\to N$ the retraction that fixes every point of $N$ and sends every point $H(t,z)$ to $z$.

Write
$$ W_{\varepsilon}^-=H([0,\varepsilon)\times l^-), \enskip W_{\varepsilon}^+=H([0,\varepsilon)\times (\partial N\setminus l^-)).$$
\noindent
Observe that if $\varepsilon$ is small enough, then $f(N)\cap W_{\varepsilon}^+=\emptyset$, and fix such an $\varepsilon$.
The morphism $f^-_{*,r}$ can be written as the composition of the morphism
$$f_*: H_r(N, l^-) \to H_r(\mathds{R}^d\setminus \int(W_{\varepsilon}^+),\mathds{R}^d\setminus (W_{\varepsilon}^+\cup \int(N)))$$ induced by $f$, the excision morphism
$$e: H_r(\mathds{R}^d\setminus \int(W_{\varepsilon}^+),\mathds{R}^d\setminus (W_{\varepsilon}^+\cup \int(N)))\to H_r(N\cup W_{\varepsilon}^-, W_{\varepsilon}^-)$$ and the map
$$\rho_*:   H_r(N\cup W_{\varepsilon}^-, W_{\varepsilon}^-) \to H_r(N, l^-) $$ induced by $\rho$.
 In others words, let $\sigma$ be a relative $r$-cycle of $(N,l^-)$. The decomposition principle tells us that the chain $f(\sigma)$
is homologous as a chain of $\mathds{R}^d\setminus W_{\varepsilon}^+$ to $\sigma_1+\sigma_2$ where $\sigma_1$ is a chain in $N\cup W_{\varepsilon}^-$ and $\sigma_2$ a chain in $\mathds{R}^d\setminus (N\cup  W_{\varepsilon}^+)$. If $\sigma$ represents $\kappa\in H_r(N,l^-)$, then $\rho(\sigma_1)$ represents  $f^-_{*,r}(\kappa)$.
 An analogous result shows that the map $f^+_{*,r}:H_r(N, l^+) \to H_r(N, l^+)$ induced by $f^+$ also represents
the $r$-homological discrete Conley index $\textit{h}_r(f^{-1}, X)$.

\medskip
Let us explain now the duality.
There is a non degenerated bilinear form
$$(\kappa^-, \kappa^+)\mapsto \kappa^-\cdot\kappa^+$$
on $H_{r}(N,l^-)\times H_{d-r}(N,l^+)$ that induces an isomorphism
$$ H_{d-r}(N,l^+)\to H^{r}(N,l^- )$$ defined as follows:

\begin{itemize}
\item every class $\kappa^-\in H_{r}(N,l^-)$ may be represented by a $r$-chain $\sigma^-$ of $N$ such that
$\partial \sigma^-\subset l^-\setminus \partial l^-$ and every class $\kappa^+\in H_{d-r}(N,l^+)$ by a $(d-r)$-chain $\sigma^+$ of $N$ such that $\partial \sigma^+\subset l^+\setminus \partial l^+$;
\item the algebraic intersection number $\sigma^-\cdot\sigma^+$ is well defined because $\partial\sigma^-\cap \sigma^+=\partial\sigma^+\cap \sigma^-=\emptyset$;
\item $\sigma^-\cdot\sigma^+$ depends only on $\kappa^-$ and $\kappa^+$ and therefore can be written $\kappa^-\cdot\kappa^+$.
\end{itemize}
Observe also the following:
\begin{itemize}
\item $\partial f(\sigma^-)\cap \sigma^+=\partial\sigma^+\cap f(\sigma^-)=\emptyset$ and so one can define $f(\sigma^-)\cdot\sigma^+$;
\item $\sigma^-\cdot f^{-1}(\sigma)^+$ can also be defined and one has $$f(\sigma^-)\cdot\sigma^+=d(f) \,\sigma^{-1}\cdot f^{-1}(\sigma^+).$$
\end{itemize}

\noindent
In order to get the duality result, it remains to prove that
$$f(\sigma^-)\cdot\sigma^+=f^-_{*,r}(\kappa^-)\cdot\kappa^+$$
 and
$$\sigma^-\cdot f^{-1}(\sigma^+)= \kappa^-\cdot f^{+}_{*,d-r}(\kappa^+).$$
One knows that $f(\sigma^-)$ is homologous as a chain in $\mathds{R}^d\setminus W_{\varepsilon}^+$ to $\sigma^-_1+\sigma^-_2$, where $\sigma_1^-$ is a chain in $N\cup W_{\varepsilon}^-$ and $\sigma_2^-$ a chain in $\mathds{R}^d\setminus (N\cup  W_{\varepsilon}^+)$ , and we can deduce that
$$f(\sigma^-)\cdot\sigma^+=\sigma^-_1\cdot\sigma^+=\rho(\sigma^-_1)\cdot\sigma^+= f^-_{*,r}(\kappa^-)\cdot\kappa^+,$$
because $\rho^-(\sigma^-_1)$ represents $f^-_{*,r}(\kappa^-)$. The second equality can be proven similarly.

\medskip
This duality result was originally proved by Szymczak in \cite{szymczak}.
It can be stated as follows:

\begin{theorem}[Szymczak]\label{thm:szymczak}
Let $f$ be a local homeomorphism of a manifold of dimension $d$ and $X$ a compact isolated
invariant set. Then, for any $0 \le r \le d$
$$\textit{h}_{d-r}(f, X) \cong d(f) \cdot (\textit{h}_r(f^{-1}, X))^*.$$
\end{theorem}

\subsection{On the connectedness of $\overline{N \setminus L}$}\label{subsec:connected}

The notion of filtration pair provides a useful tool to study isolated invariant sets. Despite we know
already that these pairs can always be chosen to be regular, it may
happen that the isolating neighborhood $\overline{N \setminus L}$ fails to be connected even though
$X$ is connected. The purpose of this section is to solve this issue by showing that, if $X$ is connected,
we can stick our considerations to the connected component of $\overline{N \setminus L}$ without losing any dynamical
information represented by the spectrum of the discrete homological Conley index of $X$ and $f$.
The $2$-dimensional version of what is done here can be found in \cite{lecalvezyoccozconley}.

Assume that $(N, L)$ is a regular filtration pair for $X$ and $f$ and
denote ${\cal S}=\pi_0(\overline{N\setminus L})$ the set of connected components of $\overline{N\setminus L}$.
Recall that, for every $r \ge 0$,
$\widetilde H_r(N/L) \sim H_r(N, L)$ and $\widetilde H_r(N/(\overline{N \setminus S})) \sim H_r(N, \overline{N \setminus S})$
because the pairs $(N, L)$ and $(N, \overline{N \setminus S})$ are regular for any $S \in \cal S$.
The quotient $N/L$ is the wedge sum of the pointed spaces $N/(\overline{N \setminus S})$
for $S \in \mathcal{S}$ and we may write
\begin{equation*}
H_r(N,L)=\bigoplus_{S\in{\cal S}} H_r(N, \overline{N\setminus S}).
\end{equation*}
This decomposition provides a way to split the action of the map $\bar{f}_{*,r}$ in
the group $H_r(N, L)$. Given $S_0, S_1 \in \mathcal{S}$, define
$$\bar f^{S_1,S_0}_{*,r}=\pi_r^{S_1}\circ \bar f_{*,r}{}_{|H_r(N, \overline{N\setminus S_0})},$$
where $\pi_r^{S_1}:H_r(N,L)\to H_r(N, \overline{N\setminus S_1})$ is the projection parallel to
$\bigoplus_{S\not=S_1} H_r(N, \overline{N\setminus S})$. Therefore, it is possible to write

$$\bar f_{*,r}{}_{\vert H_r(N, \overline{N\setminus S_0})} = \sum_{S \in \mathcal{S}} \bar{f}^{S,S_0}_{*,r}$$
for any connected component $S_0$ of $\overline{N \setminus L}$. As a consequence, we get
\begin{equation*}
\trace(\textit{h}_r(f, X)) = \trace(\bar{f}_{*,r}) =
\sum_{S \in \mathcal{S}} \trace(\bar{f}^{S,S}_{*,r}).
\end{equation*}

Let us give another interpretation of the maps $\bar f^{S_1,S_0}_{*,r}$.
Write $R_i=\overline{N\setminus S_i}$, for $i=0,1$, and denote $\pi_{R_i}: N\to N/R_i$ the projection onto the quotient space $N/R_i$ that sends every point
 $z\in N\setminus R_i$ to itself and every point $z\in R_i$ to the point $[R_i]$.
One gets a continuous map $\bar f^{S_1,S_0} : N/R_0\to N/R_1$
that sends $[R_0]$ to $[R_1]$ and every point $z\in N\setminus R_0$ to $\pi_{R_1}(f(z))$.
The action of $\bar f^{S_1,S_0}$ on the reduced homology groups $\widetilde H_r(N/R_i)\sim H_r(N,R_i)$
is nothing but $\bar f^{S_1,S_0}_{*,r}$.

\medskip
We will often use the following fact:
\begin{lemma}\label{lem:nestedcyclesgeneral}
Let $S_0$ and $S_1$ be connected components of $\overline{N \setminus L}$ and
$\kappa \in H_r(N, \overline{N \setminus S_0})$. If $\kappa$ is represented by a relative $r$-cycle $\sigma$
 of $(N, \overline{N \setminus S_0})$, then the class $\bar f^{S_1,S_0}_{*,r}(\kappa)$
is represented by a relative $r$-cycle $\sigma'$ of $(N, \overline{N \setminus S_1})$ such that
$\sigma' \subset f(\sigma \cap S_0)$.
\end{lemma}
\begin{proof}
Since $f(\overline{N\setminus L})\subset \int(N)$ and $f(L) \cap (\overline{N \setminus L})=\emptyset$, there exists a compact neighborhood $U$
 of $\overline{N \setminus L}$ in $N$ such that $f(U) \subset N$ and $f(U \cap L) \cap U = \emptyset$. One can suppose additionally that the connected component $U_0$ of $U$ that contains $S_0$ is included in $L\cup S_0$. Let $\sigma$ be a relative $r$-cycle of $(N, \overline{N \setminus S_0})$ that represents a class $\kappa \in H_r(N, \overline{N \setminus S_0})$. By excision, one can find a $r$-cycle $\sigma_0$ of $(U_0, \overline{U_0 \setminus S_0})$, such that $\sigma_0\subset \sigma$, which
represents the class $\kappa$, as a cycle of $H_r(N, \overline{N \setminus S_0})$. Recall that $H_r(N, \overline{N \setminus S_0})$ is a subspace of $H_r(N, L)$ and that the class $\bar{f}_{*,r}(\kappa)\in H_r(N, L)$ is represented
by the relative $r$-cycle $f(\sigma_0)$ of $(N,L)$. The class $\bar f^{S_1,S_0}_{*,r}(\kappa)$ is nothing but the projection of $\bar{f}_{*,r}(\kappa)$ in $H_r(N, \overline{N \setminus S_1})$. It is the homology class of $f(\sigma_0)$, seen as relative cycle of $(N, \overline{N \setminus S_1})$. By excision again, it is represented by a $r$-cycle $\sigma'$ of $(U, \overline{U \setminus S_1})$ such that $\sigma'\subset f(\sigma_0)\cap U$.
Observe now that
$$\sigma'\subset f(\sigma \cap U_0) \cap U \subset f(\sigma \cap S_0),$$ because
$U_0\subset L\cup S_0$ and $f(U\cap L)\cap U=\emptyset$.
\end{proof}

For every word $I = \{S_k\}_{0\leq k<m} \in \mathcal{S}^m$,
we define the itinerary map
$$\bar{f}^{I}_{*,r} =
 \bar{f}^{S_0,S_{m-1}}_{*,r} \circ \bar{f}^{S_{m-1},S_{m-2}}_{*,r} \circ
\ldots \circ \bar{f}^{S_2,S_1}_{*,r} \circ \bar{f}^{S_1,S_0}_{*,r}.$$
The maximal compact invariant set which follows the itinerary defined by $I$ will be denoted
$\Inv(I) = \bigcap_{k \in \mathds{Z}} f^{-k}(S_k)$, where $\{S_k\}_{k\in\mathds{Z}}$ is the $m$-periodic extension of $I$.

\begin{proposition}\label{prop:invariantempty}
For any word $I$, if $\Inv(I) = \emptyset$ then the map $\bar{f}^{I}_{*,r}$
is nilpotent, hence, in particular, $\trace(\bar{f}^{I}_{*,r}) = 0$.
\end{proposition}

\begin{proof}
Suppose $I = \{S_k\}_{0 \le k < m}$ and write $\{S_k\}_{k \in \mathds{Z}}$ for the $m$-periodic extension of $I$.
The stable set $\Lambda^- = \bigcap_{k\geq 0} f^{-k}(S_k)$, must be empty. Otherwise,
 $\Lambda^-$ would be a non-empty compact set such that $f^{m}(\Lambda^-) \subset \Lambda^-$, hence
$\Inv(I) = \bigcap_{k\geq 0} f^{km}(\Lambda^-) \neq \emptyset$. Thus, there exists $n_0\geq 1$ such that
$\bigcap_{0\leq k\leq mn_0} f^{-k}(S_k)=\emptyset$.

Fix a class $\kappa \in H_r(N,\overline{N \setminus S_0})$ represented by a relative cycle $\sigma_0$
in $(N,\overline{N\setminus S_0})$.
By Lemma \ref{lem:nestedcyclesgeneral}, we can construct inductively a sequence of relative cycles $\{\sigma_{k}\}_{0\leq k\leq mn_0}$
such that
\begin{itemize}
\item $\sigma_{k}$ represents $\bar f^{S_{k},S_{k-1}}_{*,r} \circ \ldots
 \circ \bar f^{S_{1},S_{0}}_{*,r}  (\kappa)$ and
\item $\sigma_{k+1} \subset f(\sigma_{k} \cap S_k)$.
\end{itemize}
Then, one deduces that $f^{-mn_0}(\sigma_{mn_0} )\subset\bigcap_{0\leq k\leq mn_0} f^{-k}(S_k)=\emptyset$.
Therefore, one has $(\bar {f}^{I}_{*,r})^{n_0}(\kappa) = 0$.
\end{proof}

An itinerary which defines a non-nilpotent map must necessarily contain a non-trivial invariant set. Thus, if we assume
that $X$ is connected and $S$ denotes the unique connected component of $\overline{N \setminus L}$ which contains $X$,
 we obtain that the only itinerary followed by points of $X$ must be constant equal to $S$, hence
\begin{equation*}
\trace(\textit{h}_r(f^n, X)) =
 \trace((\textit{h}_r(f, X))^n) =
\trace((\bar{f}_{*,r})^n) =
 \sum_{I \in \mathcal{S}^n} \trace(\bar{f}^{I}_{*,r}) =
 \trace((\bar f^{S,S}_{*,r})^n).
 \end{equation*}

\noindent We have proved the following:

 \begin{proposition}
If $X$ is connected, the endomorphisms $\bar{f}_{*,r}$ and $\bar f_{*,r}^{S,S}$
are spectrum equivalent.
\end{proposition}

We will conclude this subsection by looking at a particular situation. Assume that there is a compact set $L' \subset \int_N(L)$ such that
$f(\overline{N \setminus L'})\subset N$. If $S'$ is the connected component of $\overline{N \setminus L'}$ which contains $S$,
denote $$e_r : H_r(N, \overline{N \setminus S}) \to H_r(S', \overline{S' \setminus S})$$ the excision isomorphism and
$$f_{*,r}: H_r(S', \overline{S' \setminus S})\to H_r(N, \overline{N \setminus S})$$ the map induced by $f$.
One can write $\bar{f}_{*,r}^{S,S} = f_{*,r} \circ e_r$.
It will be more convenient to deal with $\widetilde{f}_{*,r} := e_r\circ f_{*,r}$. This map being conjugate to $\bar{f}^{S,S}_{*,r}$, one can state:

\begin{proposition}\label{prop:spectrumequivalent}
If $X$ is connected, the endomorphism $\widetilde{f}_{*,r}$ belongs to
the spectrum equivalence class of the $r$-homological Conley index.
\end{proposition}

To sum up, if $X$ is connected the trace computations can be done with the map $\widetilde f_{*,r}$, which is not directly induced in general
by a filtration pair but arises as the result of considering just the connected component containing $X$. A complete description
of this map for $r = 1$ will allow us to prove a closed formula for the traces of the iterates of the $1$-homological discrete Conley index.

Finally, we state the particular case $S_0 = S_1 = S$ of Lemma \ref{lem:nestedcyclesgeneral},
with $\widetilde f_{*,r}$ instead of $\bar f_{*,r}^{S,S}$, which will be used
in the proofs as a way to keep control of the successive images of an homology class under the map $\widetilde f_{*,r}$.
\begin{lemma}\label{lem:nestedcycles}
If $\kappa \in H_r(S', \overline{S' \setminus S})$ is represented by a relative $r$-cycle $\sigma$,
the homology class $\widetilde f_{*,r}$ is represented by a relative cycle $\sigma'$ such that $\sigma' \subset f(\sigma \cap S)$.
\end{lemma}

\subsection{Invariant acyclic continua}\label{subsec:theorems}

Henceforth, we will assume that our isolated invariant set $X$ is connected and acyclic. In the previous
subsection we proved that it suffices to know the traces of the iterates of the map
$\widetilde f_{*,r}$ in order to compute the sequence $\{\trace(\textit{h}_r(f^n, X))\}_{n \ge 1}$,
they are equal.
In this subsection we provide the statement of a result from which all the theorems presented
in the introduction follow easily, hence it is the key result of the article.

\begin{theorem}\label{thm:key}
Let $f$ be a local homeomorphism of $\mathds{R}^d$ and $X$ an isolated invariant acyclic continuum.
There exists two finite maps $\varphi : J \to J$ and $\psi : J' \to J'$ and, for every $r \ge 1$, a
representative $\widetilde{f}_{*,r}$ of the spectrum equivalence class of the $r$-homological discrete Conley index of $X$ and $f$
such that:
\begin{itemize}
\item $\widetilde{f}_{*,1}$ is a reduced permutation endomorphism defined by $\varphi$.
\item $\widetilde{f}_{*,r}$ is dominated by $\psi$.
\item $\varphi$ and $\psi$ are shift equivalent.
\end{itemize}
\end{theorem}

The proof of this theorem is the content of Section \ref{sec:proofthm}. We have all ingredients
to prove Theorems A, B and C. The particular cases where $X$ is an attractor or a repeller have
already been addressed in Subsection \ref{subsec:attractorrepeller}.

\begin{proof}[Proof of Theorem A]
It is straightforward once we apply Propositions \ref{prop:reducedpermutationtrace} and \ref{prop:permutationendotrace}
to the reduced permutation endomorphism $\widetilde{f}_{*,1}$, which is spectrum equivalent to any map in the class $\textit{h}_1(f, X)$.
\end{proof}

\begin{proof}[Proof of Theorem B]
From Theorem A we get that $\trace(\textit{h}_1(f, X)) = -1$ if and only if $\varphi$ is fixed point free.
In that case, $\psi$ is also fixed point free because it is shift equivalent to $\varphi$.
The trivial remark that follows Definition \ref{def:dominatedendo} finishes the proof.
\end{proof}

\begin{proof}[Proof of necessity of Theorem C]
From equation (\ref{eq:lefschetzsimple}) we obtain that for $n \ge 1$,
\begin{equation}\label{eq:ThmC1}
i(f^n, X) = - \trace(\textit{h}_1(f^n, X)) + \trace(\textit{h}_2(f^n, X)).
\end{equation}
Theorem A and Szymczak's duality tell us that there exist two finite maps
$\varphi : J \to J$ and $\varphi' : J' \to J'$ such that
$$\trace(\textit{h}_1(f^n, X)) = -1 + \#\Fix(\varphi^n)$$
and
$$\trace(\textit{h}_2(f^n, X)) = (-1)^n(-1 + \#\Fix((\varphi')^n)).$$
Plugging these expressions into (\ref{eq:ThmC1}), we deduce that
\begin{equation*}
i(f^n,X) =
\begin{cases}
2 - \#\Fix(\varphi^n) - \#\Fix((\varphi')^n) & \text{if }n \ge 1 \text{ is odd} \\
-\#\Fix(\varphi^n) + \#\Fix((\varphi')^n) & \text{if }n \ge 1 \text{ is even.}
\end{cases}
\end{equation*}
If we denote $b_k$ and $c_k$ the number of $k$-periodic orbits of $\varphi$ and $\varphi'$, respectively, we get
\begin{equation*}
i(f^n,X) =
\begin{cases}
 2 - \sum_{k | n} k \cdot(b_k + c_k) & \text{if }n \ge 1\text{ is odd} \\
-\sum_{k | n} k \cdot(b_k - c_k) & \text{if }n \ge 1 \text{ is even.}
\end{cases}
\end{equation*}
Thus, a careful computation shows that if we define
\begin{equation}\label{eq:definitionak}
a_k =
\begin{cases}
2 - b_1 - c_1 & \text{if } k = 1, \\
-1 - b_2 + c_2 + c_1 & \text{if } k = 2,\\
- b_k - c_k & \text{if } k > 1 \text{ is odd},\\
-b_k + c_k & \text{if } k > 2 \text{ and } k/2 \text{ are even,}\\
-b_k + c_k + c_{k/2} & \text{if } k > 2 \text{ is even and } k/2 \text{ is odd}.
\end{cases}
\end{equation}
we can write $\{i(f^n, X)\}_{n \ge 1} = \sum_k a_k \sigma^k$.
 Now, it is trivial to check that Corollary \ref{cor:indexinequality} implies that $a_1 \le 1$, hence
 $b_1 + c_1 \ge 1$ and, evidently, $a_k \le 0$ for all odd $k > 1$.
 Moreover, there are only a finite number of non-zero $b_k$ and $c_k$, hence of $a_k$,
which implies that the sequence of fixed point indices must be periodic.
\end{proof}

\section{A toy model: a radial case}\label{sec:sequences}

\subsection{Proof of the results}\label{subsec:proofradial}

In this section we will sketch a proof of our results for the case in which the homeomorphism
fixes only one point and presents a radial dependence.
 This type of maps will provide us with examples which
realize all possible sequences of fixed point indices of the iterates of a map at a fixed point described in Theorem C.

The $(d + 1)$-dimensional sphere is the end compactification of $S^d \times \mathds{R}$, where
one adds the lower end $e^-$, adherent to $S^d \times (-\infty, 0]$, and the upper end
$e^+$, adherent to $S^d \times [0, + \infty)$.

Let $h$ be a homeomorphism of $S^d$ and $g : S^d \rightarrow \mathds{R}$
be a continuous map. The skew-product of $g$ and $h$
$$(z, r) \mapsto (h(z), r + g(z))$$
 induces in $S^{d+1}$ a homeomorphism $f$ which fixes the two ends.
An extra hypothesis may be added to ensure that
the origin is isolated as an invariant set. Assume that
\begin{equation}\label{eq:0isolated}
g(z) \ge 0 \Rightarrow g(h(z)) > 0 \tag{P}
\end{equation}
which evidently implies that there exists $\epsilon > 0$ so that
$$g(z) \ge - \epsilon \;\Rightarrow\; g(h(z)) \ge \epsilon.$$
Property (\ref{eq:0isolated}) implies that no discrete interior tangency is possible in
any closed $(d+1)$-ball of the form $(S^d \times (-\infty, r]) \cup \{e^-\}$,
hence it is an isolating block for $f$.
In particular, there are no fixed points in $S^d \times \mathds{R}$.
Assume additionally that the origin is neither a repelling nor an attracting fixed point, which
means that $g$ must take positive and negative values.

Let $l^-$ be a submanifold with boundary of $S^d$ which is a neighborhood of $\{ z\in S^d\,\vert\, g(z)\geq \varepsilon\}$ included in $\{ z\in S^d\,\vert\, g(z) > 0\}$. Observe that
\begin{itemize}
\item $h(l^-) \subset \int(l^-)$.
\item $g > 0$ in $l^-$ and $g < 0$ in $S^d \setminus h^{-1}(l^-)$.
\end{itemize}
Define $l^+ = \overline{S^d \setminus l^-}$, then $h^{-1}(l^+) \subset \int(l^+)$.
After identifying $S^d$ to $S^d \times \{0\}$, the sets $l^-$ and $l^+$ may be considered
as subsets of $S^d \times \{0\} \subset S^d \times \mathds{R}$. If we set $N = (S^d \times (-\infty, 0]) \cup \{e^-\}$,
the triple $(N, l^-, l^+)$ is a filtration triple as defined in Subsection \ref{subsec:duality}.

The map $\bar {f}$ induced by $f$ on the quotient $N/l^-$ induces an endomorphism
$\bar{f}_{*,r} : H_r(N, l^-) \to H_r(N, l^-)$ which is a representative of the $r$-homological Conley index.
Each space $H_r(N)$ being trivial if $r\not=0$ and $1$-dimensional if $r=0$,
the connecting map
$$\partial_r: H_r(N,l^-) \to H_{r-1}(l^-),$$ induces an isomorphism between $H_r(N,l^-)$ and the reduced $r$-homological group  $\widetilde  H_{r-1}(l^-)$, where
  \begin{equation*}
\widetilde  H_{r}(l^-)= \begin{cases}
H_{r}(l^-)& \text{if} \enskip  r\geq 1 \\
  \ker (j_*) &\text{if} \enskip  r=0
\end{cases}
\end{equation*}
 and $$j_* : H_0(l^-)\to H_0(N)$$ is the inclusion-induced map.

There is an easy way to understand the inverse of the connecting map. Denoting $\Delta_r$ the standard affine simplex, one can associate to every singular $r$-simplex $\sigma:{\Delta_r}\to l^-$ a singular $(r+1)$-simplex $p(\sigma) : {\Delta_{r+1}}\to N$ defined as follows:
 \begin{equation*}
 p(\sigma)(t_1, \dots t_{r+1})= \begin{cases}
\left(  \sigma\left( \frac{t_1}{1-t_{r+1}}, \dots, \frac{t_r}{1-t_{r+1}}\right), \frac{t_{r+1}}{t_{r+1}-1},\right) & \mathrm{if} \enskip t_{r+1}\not =1, \\
e^- &\mathrm{if} \enskip  t_{r+1}=1.
\end{cases}
\end{equation*}
By linear extension, $p$ associates to every $r$-chain $\sigma$ of $l^-$ a $(r+1)$-chain of $N$. If $\sigma$ is a cycle of $l^-$ (inducing an element of $\widetilde  H_{0}(l^-)$ if $r=0$), then $p(\sigma)$ is a relative cycle of $(N,l^-)$. If $\sigma$ is the boundary of a $(r+1)$-chain of $l^-$, then $p(\sigma)$ is the boundary of a relative $(r+2)$-chain of $(N,l^-)$. The morphism $p_*:  \widetilde H_{r}(l^-)\to H_{r+1}(N,l^-)$ naturally induced is nothing but the inverse of $\partial_r$. Observe now that $$\bar{f}_{*,r} ([p(\sigma)])=[p(h(\sigma))].$$ In other words, the map $\bar{f}_{*,r+1} : H_{r+1}(N, l^-) \to H_{r+1}(N, l^-)$ is conjugate to $h_{*,r} :  \widetilde H_r(l^-) \to  \widetilde H_r(l^-)$ by $p_*$. Similarly the map $(\overline{f^{-1}})_{*,r+1} : H_{r+1}(N, l^+) \to H_{r+1}(N, l^+)$ is conjugate to $(h^{-1})_{*,r} :  \widetilde H_r(l^+) \to  \widetilde H_r(l^+)$.

The duality explained in Subsection \ref{subsec:duality} can be deduced from classical duality results. By Alexander's duality, there is a natural isomorphim
$$\widetilde  H_{r}(l^-)\to \widetilde  H^{d-r-1}(l^+),$$ where $\widetilde  H^{d-r-1}(l^+)$ is the dual space of $\widetilde  H_{d-r-1}(l^+)$. This isomorphism conjugates $h_{*,r}$ to the dual map of $(h^{-1})_{*,d-r-1}$ if $h$ preserves the orientation and to its opposite if $h$ reverses the orientation. This morphism induces naturally an isomorphism
$$H_{r+1}(N,l^-)\to \widetilde  H^{d-r}(N,l^+),$$where $\widetilde  H^{d-r}(N,l^+)$ is the dual space of $\widetilde  H_{d-r}(N,l^+)$, which conjugates $\bar{f}_{*,r+1}$ (up to the sign) to the dual map of $(\overline{f^{-1}})_{*,d-r}$.

Theorem \ref{thm:key} is obvious in this case. Here $\psi=\varphi$ is the natural map induced by $h$ on the set $\pi_0(l^-)$ of connected components of $l^-$.
Indeed, every connected component $c$ of $l^-$ is associated to an element $[c]$ of
the homology group $H_0(l^-)$, and the set $\{[c]\}_{c \in \pi_0(l^-)}$
is a basis. Taking $[N]$ as a basis of $H_0(N)$, one has the natural identification $H_0(N) \sim \mathds{Q}$. The map $j_*$ sends every $[c]$ onto $1$. This means that $h_{*,0}{}_{\vert \widetilde H_{0}(l^-)}$ is nothing but the reduced permutation endomorphism defined by $\varphi$. Moreover, for every $r\geq 1$, one has
$$\widetilde{H}_{r}(l^-)= {H}_{r}(l^-) = \bigoplus_{c \in \pi_0(l^-)} {H}_{r}(c)$$
and $h_{*,r}({H}_{r}(c))\subset {H}_{r}(\varphi (c))$.

As a conclusion, note that given a homeomorphism $h$ of $S^d$ and an attractor/repeller regular decomposition $(l^-, l^+)$
of $S^d$ it is possible to define a map $g : S^d \to \mathds{R}$  such that $g > 0$ in $l^-$ and $g < 0$ in $h^{-1}(l^+)$. The triple $(N, l^-, l^+)$
is a filtration triple for the map $f$, induced in $S^{d+1}$ by the skew-product of $g$ and $h$.

\subsection{Realizing all possible sequences $\{i(f^n,p)\}_{n \ge 1}$}

In the second half of this section we will prove the sufficiency condition of Theorem C, which characterizes the sequences $I = \{I_n\}$
of integers that can be realized as the sequence $\{i(f^n, p)\}_{n \ge 1}$ of an orientation-reversing
local homeomorphism $f$ of $\mathds{R}^3$ with a fixed point $p$ isolated as an invariant set.
The class of radial homeomorphisms already defined contains examples of
 homeomorphisms which realize any sequence $I$ satisfying the conditions in the statement
of Theorem C. The idea will be to control the sequence of fixed point indices
in terms of the combinatorial descriptions we have obtained. In order to define the examples
we need orientation-reversing homeomorphisms of $S^2$ with an arbitrary number of periodic orbits.
More precisely, we need to prove the following lemma.

\begin{lemma}\label{lem:arbitraryorbits}
Let $\varphi : J \to J$ be a permutation of the finite set $J$. It is possible to construct
an orientation-reversing homeomorphism $f$ of $S^2$ which permutes a family $\{D_j\}_{j\in J} $ of pairwise disjoint closed disks such that
\begin{itemize}
\item $f(D_j)=D_{\varphi(j)}$, for every $j\in J$;

\item if $f^n(D_j)=D_{j}$ and $n$ is even then $f^n\vert_{D_j}$ is equal to the identity;

\item if $f^n(D_j)=D_{j}$ and $n$ is odd then $f^n\vert_{D_j}$ is conjugate to the map $z\mapsto \bar z$ defined on $\mathds{D}=\{z\in \mathds{C}\, \vert\,\vert z \vert\leq 1\}$.
\end{itemize}
\end{lemma}

Several approaches may be taken to prove this result. In this article we will follow the ideas of Homma, see \cite{homma}.
Define a \emph{simple triod} in $S^2$ as the union of three closed arcs having only one point in common, endpoint of every one of them.

\begin{theorem}[Homma]\label{thm:homma}
Let $F$ be a compact, connected and locally connected subset of $S^2$. A one-to-one continuous map
from $F$ to $S^2$ can be extended to an orientation-preserving homeomorphism of $S^2$ if and only if
it preserves the cyclic order of all simple triods contained in $F$.
\end{theorem}

\begin{proof}[Proof of Lemma \ref{lem:arbitraryorbits}]
Let $ \{D_j\}_{j \in J}$ be a family of pairwise Euclidean circles of same radius, all centered on the real line $\mathds{R}$ in the complex plane $\mathds{C}$.
Let $h: \bigcup_{j\in J} D_j\to \bigcup_{j\in J} D_j$ be the map that translates each disk $D_j$ over $D_{\varphi(j)}$.  It obviously preserves the cyclic order of all simple triods. By Homma's theorem, one can extend it to a map (with the same name) on the Riemann sphere  $\mathds{C}\cup\{\infty\}$. Composing $h$ with the extension of the complex conjugation to the Riemann sphere yields the required map.
\end{proof}

Therefore, using Lemma \ref{lem:arbitraryorbits} we can produce examples of orientation-reversing
homeomorphisms of the 2-sphere having an arbitrary number of periodic orbits with prescribed periods.
Now, we are ready to complete the proof of Theorem C.

\begin{proof}[Proof of sufficiency of Theorem C]
Consider the following system of equations:
\begin{equation*}
a_k = \begin{cases}
2 - b_1 - c_1 & \text{if } k = 1, \\
-1 - b_2 + c_2 + c_1 & \text{if } k = 2, \\
- b_k - c_k & \text{if } k > 1 \text{ is odd}, \\
-b_k + c_k & \text{if } k > 2 \text{ and } k/2 \text{ are even,}\\
-b_k + c_k + c_{k/2} & \text{if } k > 2 \text{ is even and } k/2 \text{ is odd}.
\end{cases}
\end{equation*}
Clearly, there exist two sequences of non-negative integers $\{b_k\}_{k \ge 1}$ and
$\{c_k\}_{k \ge 1}$ which satisfy the system of equations and such that $c_1 = 1$, $b_2 \ge 1$ and $c_k = 0$ for all odd $k > 1$.
Since at most a finite number of $a_k$ are non-zero, we can assume that there are only a finite number
of non-zero integers $b_k$ and $c_k$ as well.
By applying Lemma \ref{lem:arbitraryorbits}, choose an orientation-reversing homeomorphism $h^+$ of a 2-sphere $S^+$ having
$b_k$ cycles of pairwise disjoint disks of period $k$, for every positive integer $k \neq 2$,
which include a couple of 2-periodic disks $\{D^+_1,D^+_2\}$.
This set of periodic disks, excluding $D^+_1$ and $D^+_2$, will be denoted $\mathcal{F}^+$.
Similarly, one may construct an orientation-reversing homeomorphism $h^-$ of a 2-sphere $S^-$
that induces a permutation on a set of pairwise disjoint closed disks, which will be denoted $\mathcal{F}^-$, with $c_k$ cycles
of length $k / 2$ for every even number $k \ge 2$, plus an additional fixed disk $D^-$.
The complement of $D^-$ and all other periodic disks in $\mathcal{F}^-$ will be denoted $\Sigma^-$. One can suppose that our two homeomorphisms satisfy the last two assertions of Lemma \ref{lem:arbitraryorbits}.
Consider two spheres $S^-_1$ and $S^-_2$ of the type $S^-$ together with $S^+$.
Let $s$ be a map which identically identifies $S^-_1$ to $S^-_2$ and viceversa.
Take out the interior of the disks
$D^-$ in $S^-_1$ and $S^-_2$ and the disks $D^+_1$ and $D^+_2$ of $S^+$. Then, paste one boundary $\partial D^-$
to each $\partial D^+_i$, $i = 1,2$. The result is a topological sphere.
The fact that our two homeomorphisms satisfy the assertions of Lemma \ref{lem:arbitraryorbits}
implies that the pasting can be done in such a way that there exists an homeomorphism $h'$  of our sphere coinciding with $h^+$ in the complement of $D^+_1 \cup D^+_2$ and with the composition $s \circ h^- = h^- \circ s$
in the complement of the two disks $D^-$ in $S^-_1 \cup S^-_2$. Construct a homeomorphism $h$ by
composing $h'$ with an orientation-preserving homeomorphism $h''$ such that every disk $D$ of $\mathcal{F}^+$
satisfies $h''(D) \subset \int(D)$ and $h''(\Sigma^-_1) \subset \int(\Sigma^-_2)$ and $h''(\Sigma^-_2) \subset \int(\Sigma^-_1)$.
The union $l^-$ of the disks in $\mathcal{F}^+$, $\Sigma^-_1$ and $\Sigma^-_2$ is an attracting set for $h''$ and, more importantly, for $h$.

The action $h_{*,0}$ of $h$ on the reduced homology group $\widetilde{H}_0(l^-)$ is associated to the reduced permutation automorphism defined by a permutation having $b_k$ cycles of length $k$, for all $k \ge 1$.
The action $h_{*,1}$ of $h$ on $\widetilde{H}_1(l^-)={H}_1(l^-)$ is associated to the opposite of the permutation
automorphism associated to a permutation with $c_k$ cycles of length $k$, for all even $k \ge 2$.
Indeed, the boundaries of the disks of the sets $\mathcal{F}^-$ corresponding to each $S^-_1$ and $S^-_2$ define a basis of ${H}_1(l^-)$.
To study $h_{*,1}$, one can also use the duality  by considering the repeller  $l^+ = \overline{S^2 \setminus l^-}$ and look at the action of $(h^{-1})_{*,0}$ on $\widetilde{H}_0(l^+)$ which is the reduced permutation automorphism defined by a permutation having $c_k$ cycles of length $k$, for all $k \ge 1$.
The traces of the maps $(h_{*,0})^n$ and $(h_{*,1})^n$ are now easy to compute using Propositions \ref{prop:reducedpermutationtrace},
\ref{prop:permutationendotrace} and Szymczak's Duality.
They are equal respectively to $-1+\sum_{k|n} k \cdot b_k$ and to $(-1)^n(-1+ \sum_{k|n} k \cdot c_k)$.
Combining these expressions we obtain,
\begin{equation*}
-\trace((h_{*,0})^n) + \trace((h_{*,1})^n) =
\begin{cases}
2 -\sum_{k | n} k \cdot (b_k + c_k) & \text{if }n \text{ is odd,} \\
-\sum_{k | n} k \cdot (b_k - c_k) & \text{if }n \text{ is even.}\\
\end{cases}
\end{equation*}

As remarked in the end of Subsection \ref{subsec:proofradial}, it is possible
to associate to the attractor/repeller pair $(l^-, l^+)$ a continuous map
$g : S^2 \to \mathds{R}$ such that $g > 0$ in $l^-$ and $g < 0$ in $h^{-1}(l^+)$.
The map $f$ induced in $S^3$ by the skew-product of $h$ and $g$,  is the one
we are looking for. Summarizing what has been done in Subsection \ref{subsec:proofradial},
one knows that the lower end $\{e^-\}$ is an isolated invariant set for $f$ and, for any integer $n \ge 1$,
$$i(f^n, e^-) = - \trace((\bar{f}_{*,1})^n) + \trace((\bar{f}_{*,2})^n)=- \trace((h_{*,0})^n) + \trace((h_{*,1})^n),$$
whose exact value has been computed in terms of the integers $b_k$ and $c_k$.
It remains to realize that the work done in equation (\ref{eq:definitionak}) guarantees that the definition of $b_k$ and $c_k$ leads to the expression
$$\{i(f^n, e^-)\}_{n \ge 1} = \sum_k a_k \sigma^k.$$
\end{proof}

\section{Proof of Theorem \ref{thm:key}}\label{sec:proofthm}

\subsection{Construction of a good filtration pair}

We suppose in this section that $X$ is an isolated invariant acyclic continuum of a local homeomorphism $f$.
Consider a regular filtration pair $(N, L_0)$ for $X$ and $f$ such that $N$ is connected.
The compact set $L'_1=(f^{-1}(L_0)\cap N)\cup L_0$ is a neighborhood of $L_0$ in $N$ such that
$$f(L'_1)\cap (\overline {N\setminus L'_1})\subset (L_0\cap (\overline {N\setminus L'_1}))\cup (f(L_0)\cap (\overline {N\setminus L'_1}))=\emptyset$$
and also that $\overline {N\setminus L'_1}$ is an isolating neighborhood of $X$.
 Therefore, we can find a neighborhood $L_1$ of $L'_1$ in $N$ such that
  \begin{itemize}
\item $(N,L_1)$ is a regular pair,
\item $\overline {N\setminus L_1}$  is an isolating neighborhood of $X$,
\item $f(L_1)\cap (\overline {N\setminus L_1})=\emptyset$,
\item $f(\overline{N\setminus L_1})\subset N\setminus L_0$,
\end{itemize}
In particular, $(N, L_1)$ is a filtration pair.
With the same process, we can define inductively a sequence $(N,L_n)$ of regular filtration pairs of $X$, such that
 \begin{itemize}
 \item $L_{n+1}$ is a neighborhood of $(f^{-1}(L_n)\cap N)\cup L_n$ in $N$,
\item $f(\overline{N\setminus L_{n+1}})\subset N \setminus L_n$.
\end{itemize}
Note that these properties imply that $f(\overline{ L_{n+1}\setminus L_n})\subset \int_N(L_n)$.
Observe that the boundary
of a set $\overline{N \setminus L_{n+1}}$ relative to $N$ and relative to any $\overline{N \setminus L_m}$, $m\leq n$ coincide
 because $\partial_N L_{n+1} \cap L_{n} = \emptyset$ and also that
$f(\partial_N L_n) \cap (\overline{N \setminus L_n}) = \emptyset$.

Let us denote $S_n$ the connected component of $\overline{N\setminus L_n}$ that contains $X$.
Of course one has that $f(S_n)\subset S_{m}$, whenever $0 < m < n$.
Indeed $f(S_n)$ is connected, contains $X$ and does not meet $L_{m}$.
It is trivial to check that each triple $(N, S_m, S_n)$ is regular and satisfies
the properties
\begin{enumerate}
\item[i)] $S_m$ is a neighborhood of $S_n$ in $N$ (for the relative topology),
\item[ii)] $f(S_m)\subset N$, $f(S_n) \subset S_m$,
\item[iii)] $f(\partial_N S_n)\cap S_n =\emptyset$,
\item[iv)]$S_n$ is an isolating neighborhood and $\Inv(S_n) = X$,
\end{enumerate}
introduced in Subsection \ref{subsec:connected}.

For every $n\geq 0$ and every $r\in\{0,\dots, d\}$ denote $E_{r,n}$ the subspace of $H_r(N)$
 generated by the $r$-cycles in $S_n$. We obtain $d+1$ non-increasing sequences.
 Each space $H_r(N)$ being finite-dimensional, one can find an integer $n_0$ such
  that for every $n\geq n_0$ and every $r\in\{0,\dots,d\}$, one has $E_{r,n}=E_{r,n_0}$. Replacing $L_0$ with $L_{n_0}$, and each $L_n$ with $L_{n_0+n}$, we have the following property, consequence of the minimality of the images of $H_r(S_n) \rightarrow H_r(N)$ and the following chain of inclusions:
$$S_{n+m}  \subset \bigcap_{0\leq k\leq m} f^{-k}(S_n) \subset S_{n}.$$

\begin{enumerate}
\item[v)] For every $n, m\geq 0$ and $r\in\{0,\dots,d\}$, every $r$-cycle in $S_n$ is homologous, as a cycle of $N$, to a $r$-cycle
in $\bigcap_{0\leq k\leq m} f^{-k}(S_n)$.
\end{enumerate}

Let us fix some extra notation valid for the rest of the proof:
$$
 S' = S_0, \;\; S = S_1.
$$

The endomorphisms
$$\widetilde{f}_{*,r} : H_r(S',\overline {S'\setminus S}) \rightarrow
H_r(S',\overline {S'\setminus S})$$
induced by $f$ were shown, in Proposition \ref{prop:spectrumequivalent}, to be
spectrum equivalent to the $r$-homological discrete Conley index of $X$ and $f$.
These will be the representatives required by Theorem \ref{thm:key}.

\subsection{Nilpotence}\label{subsec:nilpotence}

A first approach to the task of describing the homological Conley indices
may wonder about the homology groups of the sets $S_n$,
and, in particular, to that of $S'$. An isolating neighborhood as $S'$ that arises
from a filtration pair should be somehow similar to the invariant set $X$. Thus, for example,
it may be possible to prove that there exists a contractible $S'$ provided that
$X$ is reduced to a point. However, we have not found yet the dynamical argument, if any,
which allows to make such a simplification. The discussion which we present in this
subsection formalizes the following idea: a homology class of $S'$ which is
not represented by some chain close to the invariant set $X$, is eventually mapped onto
the zero class by the map $\widetilde{f}_{*,r}$.

The idea of nilpotence is already present in the work of Richeson and Wiseman, see \cite{richesonwiseman}.
In our context, we have to use it in a much more delicate way and need property v).

 Denote $F_r$ the image of the inclusion map
$$\iota_r: H_r(S')\to H_r(S',\overline {S'\setminus S}),$$
 that is, the subspace of
$H_r(S',\overline {S'\setminus S})$ generated by the $r$-cycles in $S'$.

\begin{proposition}\label{prop:nilpotence}
The space $F_r$ is forward invariant under $\widetilde f_{*,r}$ and included in its generalized kernel.
\end{proposition}

\begin{proof} Let us begin by proving the invariance of $F_r$. By property v) every $r$-cycle in $S'$ is homologous, as a cycle of $N$,
 to a $r$-cycle of $S' \cap f^{-1}(S')$, so it is homologous to such a
 cycle as a relative cycle of $(N,\overline{N \setminus S})$ and, by excision,
  as a relative cycle of $(S',\overline {S'\setminus S})$.
In particular, every homology class in $F_r$ is represented by a $r$-cycle $\sigma$ in
$S' \cap f^{-1}(S')$ and its image $\widetilde f_{*,r}$ is represented by $f(\sigma)$ which is a $r$-cycle in $S'$.
This means that $F_r$ is forward invariant under $\widetilde f_{*,r}$.
 Let us prove now that $F_r$ is included in the generalized kernel of $\widetilde f_{*,r}$.
Since $X$ is acyclic, we can find neighborhoods $V \subset U$ of $X$ contained in $\int(S)$
 such that the inclusion-induced map $H_r(V) \to H_r(U)$ is trivial.
 The set $\bigcap_{k\in\mathds{Z}} f^{-k}(S')$ being reduced to $X$,
 there exists $n_0\geq 0$ such that
$$\bigcap_{\vert k\vert\leq n_0} f^{-k}(S')\subset V.$$
By using again property v) one knows that every class $\kappa\in F_r$ is represented by a $r$-cycle $\sigma$ in $\bigcap_{k = 0}^{2n_0} f^{-k}(S')$.
This implies that $\widetilde f_{*,r}^{n_0}(\kappa)$ is represented by $f^{n_0}(\sigma)$, which is a cycle in $V$, hence a boundary in $U$.
One deduces that $\widetilde f_{*,r}^{n_0}(\kappa)=0$.
\end{proof}

\subsection{Definition of $\varphi$, description of $\widetilde{f}_{*,1}$}\label{subsec:proof1}

Property v) applied to $r = 1$ gives information about the display of the subsets $S$ and $S'$ of $N$, which
the next proposition will illustrate.

\begin{lemma}\label{lem:uniquecc}
Let $(M, T', T)$ be a regular triple such that all three sets are connected,
\begin{itemize}
\item $T \cap (\overline{M \setminus T'}) = \emptyset$ and
\item the images of the inclusion-induced maps $H_1(T) \rightarrow H_1(M)$ and $H_1(T') \rightarrow H_1(M)$ are equal.
\end{itemize}
Then:
\begin{itemize}
\item Given any $c \in \pi_0(\overline{T' \setminus T})$, the set $c\cap T=c \cap \partial_M T$ is connected.
\item If $\Delta \subset T$ is closed and $\Delta \cap \partial_M T = \emptyset$, there is a bijection
$$\lambda : \pi_0(T' \setminus \Delta) \rightarrow \pi_0(T \setminus \Delta)$$
defined by $\lambda(c) \subset c$ for any $c \in \pi_0(T' \setminus \Delta)$.
\end{itemize}
\end{lemma}
\begin{proof}
Let us begin by proving the first point. Every connected component of $\partial_M T$ is a
$(d-1)$-manifold with boundary, whose boundary belongs to $\partial M$. It defines a homology class
$\kappa \in H_{d-1}(M, \partial M)$ and, by Lefschetz duality, a cohomology class in $H^1(M)$,
defined by intersecting $\kappa$ with homology classes in $H_1(M)$. This cohomology class vanishes on the
image of $H_1(T)$ in $H_1(M)$, so it vanishes on the image of $H_1(T')$ in $H_1(M)$, by hypothesis.
Every connected component $c$ of $\overline {T'\setminus T}$ meets $\partial_M T$ because $T'$ is connected
and one knows that $c\cap \partial_M T=c \cap T$. More precisely, $c\cap \partial_M T$
is the finite union of connected components of $\partial_M T$ contained in $c$.
The manifold $T$ being connected, if there is more than one component, one can find a loop $\gamma$
in $T'$ that intersects a given component of $\partial_M T$ in a unique point and transversally.
The cohomology class defined by this component does not vanish on the homology class of $\gamma$.

Let us prove now the second point. Denote $\mu: \pi_0(T \setminus \Delta)\to\pi_0(T '\setminus \Delta)$
 the map that assigns to every connected component $C$ of $T \setminus \Delta$ the connected component of
 $T '\setminus \Delta$ that contains $C$. The first point tells us that $\mu(C)$ is the union of $C$ and
of the connected components $c$ of $\overline{T'\setminus T}$ such that the non-empty connected set
$c\cap\partial_M T$ is included in $C$. As a consequence, one deduces that $\mu$ is one-to-one.
The fact that $\mu$ is onto is an immediate consequence of the connectedness of $T'$. One has $\lambda =\mu^{-1}$.

\end{proof}

The previous lemma permits us to define a map
$$\varphi: \pi_0(\overline{S' \setminus S})\to \pi_0(\overline{S' \setminus S})$$ as follows. Take a component $c \in \pi_0(\overline{S' \setminus S})$. One knows that
$c \cap S = c \cap \partial_N S$ is connected. Since $f(S)\subset S'$ and $f(\partial_{N} S) \cap S = \emptyset$, one deduces that
$f(c \cap \partial_N S)$ is contained in a connected component $\varphi(c) \in \pi_0(\overline{S'\setminus S})$.

\begin{proposition}
The spectrum equivalence class of $\textit{h}_1(f, X)$
 is represented by the reduced permutation endomorphism defined by $\varphi$.
\end{proposition}
\begin{proof}

By Lemma \ref{lem:gimgker} and Proposition \ref{prop:nilpotence}, one knows that the induced endomorphism
 $$\check{f}_{*,1} : H_1(S',\overline {S'\setminus S})/F_1 \rightarrow
H_1(S',\overline {S'\setminus S})/F_1$$
is shift equivalent to $\widetilde{f}_{*,1}$ and therefore represents
the spectrum equivalence class of the first-homological Conley index of $X$ and $f$.

From the long exact sequence of homology groups of the pair $(S', \overline{S' \setminus S})$,
the homology group $H_1(S',\overline {S'\setminus S})/F_1$ is isomorphic to $\im(\partial_1) =  \ker(j_*)$,
 where $$\partial_1 : H_1(S',\overline {S'\setminus S}) \rightarrow H_0(\overline {S' \setminus S})$$
is the connecting map and
$$j_* : H_0(\overline{S' \setminus S})
\to H_0(S')$$
the inclusion-induced map.

Every connected component $c$ of $\overline{S' \setminus S}$ is associated to an element $[c]$ of
the homology group $H_0(\overline{S' \setminus S})$, and the set $\{[c]\}_{c \in \pi_0(\overline{S' \setminus S})}$
is a basis. To the map $\varphi$ is naturally associated a permutation endomorphism $u$ on $H_0(\overline{S' \setminus S})$. Taking $[S']$ as a basis of $H_0(S')$, one has the natural identification $H_0(S') \sim \mathds{Q}$. The map $j_*$ sends every $[c]$ onto $1$. Observe now that the isomorphism $H_1(S',\overline {S'\setminus S})/F_1\to \ker(j_*)$ conjugates $\check{f}_{*,1}$ to the restriction of $u$ to $\ker(j_*)$ which is nothing but the reduced permutation endomorphism defined by $\varphi$.
\end{proof}

The previous proposition proves the first item of Theorem \ref{thm:key}.

\subsection{Introducing the stable set}

Recall that the stable set of $f$ in $S$ is $\Lambda^+=\bigcap_{k\geq 0} f^{-k}(S)$.
Lemma \ref{lem:uniquecc} implies the following:
 \begin{itemize}
\item every connected component of $S'\setminus\Lambda^+$ that meets a connected component of $\overline{S'\setminus S}$  contains this component.
\end{itemize}
Notice that every connected component of $S'\setminus\Lambda^+$ contains a unique
connected component of $S \setminus \Lambda^+$.
Denote $ {\cal C}=\pi_0(S'\setminus \Lambda^+) $ the set of connected components of
$S'\setminus \Lambda^+$ and ${\cal C}^*$ the set of components $C'\in{\cal C}$ that meet
 $\overline{S'\setminus S}$. Like in Subsection \ref{subsec:proof1}, one can define a map
$$\Psi:{\cal C}\to {\cal C}$$ as follows.
Consider a connected component $C'\in{\cal C}$ and denote $C$ the unique
connected component of $S \setminus \Lambda^+$ contained in $C'$.
Every point in $S$ whose image by $f$ belongs to $\Lambda^+$ must be contained in $\Lambda^+$.
One deduces that $f(C)$ is included in $S' \setminus \Lambda^+$, hence contained in a unique connected component $\Psi(C') \in \mathcal{C}$.
Should $C'$ meet $\overline{S' \setminus S}$, then $C'$ would meet
$\partial_N S$ so, by definition, $\Psi(C')$
would also meet $\overline{S' \setminus S}$ because $f(\partial_N S) \subset \overline{S' \setminus S}$.
Therefore, $\mathcal{C}^*$ is forward invariant.
Observe that, for every $C'\in{\cal C}$, there exists $n\geq 0$ such that
\begin{itemize}
\item  $f^n(C')\cap (\overline{S'\setminus S})\not=\emptyset$;
\item $f^k(C')\cap (\overline{S'\setminus S})=\emptyset$, if $k<n$.
 \end{itemize}
In particular, $\Psi^k(C')\in{\cal C}^*$ for every $k\geq n$.

\begin{proposition}\label{prop:gimmeetstable}
Suppose that for some $c, c' \in \pi_0(\overline{S' \setminus S})$ and any $n \ge 0$,
$\varphi^n(c) \neq \varphi^n(c')$. Then, every path in $S'$ which joins $c$ and $c'$ meets $\Lambda^+$.
\end{proposition}
\begin{proof}
 Let
$\gamma_0$ be a path in $S'$ which joins $c$ and $c'$. As a 1-cycle in $(S', \overline{S' \setminus S})$, it represents a relative homology class
$[\gamma_0] \in H_1(S', \overline{S' \setminus S})$. With the notations introduced in Subsection \ref{subsec:proof1}, one knows that
$\widetilde{f}_{*,1}^n ([\gamma_0])\not=0$ for every $n\geq 1$, because
$$\partial_1(\widetilde{f}_{*,1}^n ([\gamma_0])) = [\varphi^n(c)] - [\varphi^n(c')]\not=0.$$
\noindent
By Lemma \ref{lem:nestedcycles}, one can construct a sequence of relative 1-cycles $\{\gamma_n\}_{n \ge 0}$ of $(S',\overline{S'\setminus S})$
such that
\begin{itemize}
\item $\gamma_n$ represents the class $\widetilde{f}_{*,1}^n ([\gamma_0])$ and
\item $\gamma_{n+1} \subset f(\gamma_n \cap S)$, for every $n \ge 0$.
\end{itemize}
The sequence $\{f^{-n}(\gamma_n\cap S)\}_{n\geq 0}$ is a decreasing sequence of non-empty compact subsets of $\gamma_0$. Therefore the set $\bigcap_{n \ge 0} f^{-n}(\gamma_n \cap S)$ is not empty, included both in $\gamma_0$ and in $\bigcap_{n \ge 0} f^{-n}(S)=\Lambda_+$.
\end{proof}

Since every $C \in \mathcal{C}$ is pathwise connected, the previous proposition shows that any two
components of $\overline{S'\setminus S}$ that are contained in the same component $C \in \mathcal{C}^*$
have equal image under $\varphi^n$ for sufficiently large $n$.

\begin{corollary}\label{cor:varphiPsi}
The maps $\varphi$ and $\psi=\Psi\vert_{{\cal C}^*}$ are shift equivalent.
\end{corollary}
\begin{proof}
The inclusion $\Lambda^+ \subset S$ induces a map $a : \pi_0(\overline{S' \setminus S}) \rightarrow {\cal C}^*$ which is onto and
semiconjugates $\varphi$ and $\Psi$. This last property can be deduced immediately from the following fact: for every component $c\in\pi_0(\overline{S'\setminus S})$, the non-empty connected set $f(c\cap \partial_N S)$ belongs both to $\varphi(c)$ and to $\psi(a(c))$. As seen in Section \ref{sec:preliminaries},
 there exists an integer $n_0$ such than the restriction of $\varphi$ to $\im (\varphi^{n_0}) $
 is the permutation induced by $\varphi$. The map $a$ being onto, its restriction to $\im (\varphi^{n_0})$
sends $\im (\varphi^{n_0})$  onto $\im (\psi^{n_0})$. Proposition \ref{prop:gimmeetstable} tells us that the restriction of $a$ to
 $\im (\varphi^{n_0})$ is one-to-one, which implies that this restriction conjugates
$\varphi^{n_0}\vert_{\im (\varphi^{n_0})}$ to $\psi^{n_0}\vert_{\im (\psi^{n_0})}$.
 Consequently, the restriction of $\psi$ to $\im (\psi^{n_0})$ is bijective.
This implies that it is the permutation induced by $\psi$.
The permutation induced by $\varphi$ and $\psi$ being conjugate, $\varphi$ and $\psi$ are shift equivalent.
\end{proof}

This corollary proves the third item of Theorem \ref{thm:key}.

\subsection{Gathering all dynamical information around the unstable set}

Given a neighborhood $U$ of $X$ contained in the interior of $S$, denote $E_r(U)$ the subspace of
$H_r(S',\overline {S'\setminus S})$ generated by the relative $r$-cycles of
$(S',\overline {S'\setminus S})$ that are included in
$(S'\setminus \Lambda^+)\cup U$. The space $H_r(S',\overline {S'\setminus S})$
 being finite-dimensional, one can choose $U$ such that for every neighborhood
 $V$ of $X$ contained in $U$ one has $E_r(V)=E_r(U)$.

\begin{proposition}\label{prop:gimunstableset}
The space $E_r(U)$ is forward invariant under $\widetilde f_{*,r}$ and contains the generalized image of
$\widetilde f_{*,r}$.
\end{proposition}

\begin{proof}

Fix a neighborhood $V$ of $X$ such that $U$ contains $V$ and $f(V)$.
Every homology class $\kappa\in E_r(U)$ is represented by a relative cycle $\sigma$ of $(S',\overline {S'\setminus S})$ included in
$(S'\setminus \Lambda^+)\cup V$ because $E_r(V) = E_r(U)$. The class $\widetilde f_{*,r}(\kappa)$  is represented  by a relative cycle of $(S',\overline {S'\setminus S})$ which is included in $f(\sigma \cap S)$. But such set will meet $\Lambda^+$ only in $f(V)$
(recall that $f(S\setminus \Lambda^+)\cap \Lambda^+ =\emptyset$) so it is included in  $(S'\setminus \Lambda^+)\cup U$.
One deduces that $E_r(U)$ is forward invariant.

Let us conclude by proving that it contains the generalized image of $\widetilde f_{*,r}$.
Fix $\kappa\in H_r(S', \overline {S'\setminus S})$ and $\sigma_0$ a relative cycle in
$(S',\overline {S'\setminus S})$  representing $\kappa$.
By Lemma \ref{lem:nestedcycles}, one can construct inductively a sequence of relative cycles $\{\sigma_n\}_{n\geq 0}$ of
 $(S',\overline {S'\setminus S})$ such that
\begin{itemize}
\item $\sigma_n$ represents $\widetilde f_{*,r}^n(\kappa)$,
\item $\sigma_n\subset f(\sigma_{n-1} \cap S)$.
\end{itemize}

Consequently, one deduces that
$\sigma_n \cap S \subset \bigcap_{0\leq k\leq n} f^{k}(S)$. The fact that
$\bigcap_{k\in\mathds{Z}} f^{-k}(S) = X$ implies that
$\left(\bigcap_{0\leq k\leq n} f^{k}(S)\right)\cap \Lambda^+\subset U$
 if $n$ is large enough, hence $\sigma_n\cap \Lambda^+\subset U$.
  Therefore, one has $\widetilde f_{*,r}^n(\kappa)\in E_r(U)$.
\end{proof}

\subsection{Higher-dimensional homological decomposition}

The space $G_r(U)=E_r(U)+F_r$ is forward invariant and contains both $F_r$ and the generalized image of $\widetilde f_{*,r}$. Applying Lemma \ref{lem:gimgker} to the couple $(F_r, G_r(U))$, one knows that $\widetilde f_{*,r}$ is shift equivalent to $\check f_{*,r}$, the induced endomorphism of $G_r(U)/F_r$. We will prove now that $\check f_{*,r}$ is dominated by $\psi$. For every $C\in{\cal C}$, write $E_r^C(U)$ for the subspace of
$H_r(S',\overline {S' \setminus S})$ generated by the relative $r$-cycles of
$(S',\overline {S'\setminus S})$ that are included in
$C\cup U$ and define  $G_r^C(U)=E_r^C(U)+F_r$. The set ${\cal C}^*$ being finite, one can suppose $U$ small enough to ensure that for every neighborhood $V$ of $X$ contained in $U$, and every $C\in{\cal C}^*$, one has $E_r^C(V)=E_r^C(U)$.

\begin{proposition}
Let us suppose that $r\geq 2$. Then one has a direct sum
$$G_r(U)/F_r=\bigoplus_{C\in{\cal C^*}}G_r^C(U)/F_r.$$
 Moreover, for every $C\in{\cal C^*}$, one has
 $$\check f_{*,r}\left(G_r^C(U)/F_r\right)\subset G_r^{\psi (C)}(U)/F_r.$$
\end{proposition}

\begin{proof}
Since $X$ is acyclic there exists neighborhoods $W \subset V$ of $X$ contained in $U$ such that
the inclusion-induced maps $H_s(W) \to H_s(V)$ and $H_s(V) \to H_s(U)$ are trivial for every $s \ge 1$.
By hypothesis, $E_r^C(W) = E_r^C(U)$ for every $C \in \mathcal{C}^*$.
The proof of Mayer Vietoris formula tells us that  $$E_r(W)=\sum_{C\in{\cal C}}E_r^C(W).$$ Indeed, if $\sigma$ is a relative chain of $(S', \overline{S' \setminus S})$ included in $(S' \setminus \Lambda^+)\cup W$, the decomposition principle tells us that the chain $\sigma$ is homologous to
$$\sigma=\sum_{i\in I} \sigma_i+\sigma',$$
where $\sigma'$ is a $r$-chain in $W$ and where each $\sigma_i$ is a connected $r$-chain in
$S' \setminus \Lambda^+$ whose boundary may be written
$$\partial\sigma_i=(\partial\sigma_i)^{\overline{S'\setminus S}}+(\partial\sigma_i)^W$$
 where $(\partial\sigma_i)^{\overline{S'\setminus S}}$ is a $(r-1)$-cycle in $\overline{S'\setminus S}$
 and $(\partial\sigma_i)^W$ a $(r-1)$-cycle in $W$. Of course each cycle $\sigma_i$ is included in a connected component $C_i\in{\cal C}$. The fact that $H_{r-1}(W) \to H_{r-1}(V)$ and $H_r(V) \to H_r(U)$ are trivial implies that there exists a $r$-chain $\nu_i$ in $V$ such that $\partial \nu_i=-(\partial \sigma_i)^W$ and a $(r+1)$-chain $\omega$ in $U$ such that $\partial \omega=\sigma'-\sum_{i \in I} \nu_i$. One deduces that $\sigma$ is homologous in
 $(S', \overline{S' \setminus S})$  to $\sum_{i\in I} (\sigma_i+\nu_i)$, which implies that its homology class is included in $\sum_{C\in{\cal C}}E_r^C(U)$.

Every space $E_r^C(U)$ being included in $F_r$ if $C\not\in{\cal C}^*$, one can write
 $$G_r(U)=\sum_{C\in{\cal C}^*}G_r^C(U).$$
To prove that it is a direct sum, one must prove that  for every family $\{\kappa_C\}_{C\in{\cal C}^*}$ in
$H_r(S',\overline{S' \setminus S})$ such that $\kappa_C\in E_r^C(U)$ and $\sum_{C\in{\cal C}^*} \kappa_C\in F_r$, then one has $\kappa_C\in F_r$ for every $C\in{\cal C}^*$.
 Let us consider the connecting map
$$\partial_r: H_r(S',\overline {S'\setminus S}) \to H_{r-1}(\overline {S' \setminus S}) $$and the inclusion map
 $$\iota_r: H_r(S')\to H_r(S',\overline {S'\setminus S}).$$
 Recall that $\im(\iota_r)=\ker(\partial_r)$. We will use the equality
 $$H_{r-1}(\overline {S'\setminus S})=\bigoplus_{C\in{\cal C}^*}H_{r-1}((\overline {S'\setminus S})\cap C),$$
 true if $r\geq 2$. We have
$$\sum_{C\in{\cal C}^*}\kappa_C\in F_r\Longrightarrow \sum_{C\in{\cal C}^*} \partial_r(\kappa_C)=0
\Longrightarrow\partial_r(\kappa_C)=0\Longrightarrow \kappa_C\in \im(\iota_r)\Longrightarrow \kappa_C\in F_r$$ for every $C\in{\cal C}^*$.

A proof similar to the proof of the invariance of $E_r(U)$ gives us
$$\widetilde f_{*,r}(E_r^C(U))\subset E_r^{\psi(C)}(U),$$
for every $C\in{\cal C}^*$. Indeed, every homology class in $E_r^C(U)$, $C\in{\cal C}^*$, is represented by a relative cycle $\sigma$ included in $C\cup V$ and its image by $\widetilde f_{*,r}$ is represented  by a relative cycle included in $\psi(C)\cup f(V)$, so by a relative cycle included in $\psi(C)\cup U$.
\end{proof}

This proposition completes the proof of Theorem \ref{thm:key}.

\section{Further remarks}

The purpose of this section is to give an overview on what was known about the fixed point index
of isolated fixed points of local homeomorphisms and to show how the results contained in this article
fit in the theory.

It seems to be a connection between the behavior of the fixed point index of orientation-preserving
local homeomorphisms in $\mathds{R}^d$ and that of orientation-reversing local homeomorphisms in $\mathds{R}^{d+1}$. Let us begin with the obvious case $\mathds{R}^0=\{0\}$. The unique homeomorphism of $\mathds{R}^0$ is orientation-preserving and the index of its unique fixed point is $1$. Let us continue with the simple case of $\mathds{R}$.
In dimension $1$, a fixed point $p$ of an orientation-preserving homeomorphism $f$ isolated in the set $\Fix(f)$ is necessarily an isolated invariant set, the fixed point index can take the values $-1$, $0$ or $1$ and the sequence $\{i(f^n,p)\}_{n\geq 1}$ is constant.
A fixed point of an orientation-reversing homeomorphism is always isolated in the set of fixed points (but not necessarily an isolated invariant set) and its index is $1$.

In dimension $2$, one can construct, for every $l\in\mathds{Z}$, a local orientation-preserving homeomorphism having a fixed point $p$ such that $i(f,p)=l$. But if one supposes that $\{p\}$ is an isolated invariant set, it was proven by Le Calvez and Yoccoz in \cite{lecalvezyoccoz} that $l\leq 1$. More precisely there exist two integers $q\geq 1$ and $r\geq 1$ such that the sequence of indices of iterates can be written
$$\{i(f^n,p)\}_{n\geq 1}=\sigma_1-r \sigma_q.$$
Franks introduced Conley index techniques in \cite{franks} to get a short proof of the inequality $i(f,p)\leq 1$ and of the fact that the sequence $\{i(f^n,p)\}_{n\geq 1}$ takes periodically a non-negative value.
Some restrictions appear in the orientation-reversing case, the index of an isolated (in $\Fix(f)$) fixed point only can take the values $-1$, $ 0$ or $1$, as
was shown by Bonino in \cite{bonino}. In the case where the fixed point is an isolated invariant set, much more can be said. Ruiz del Portal and Salazar, see \cite{rportalsalazar}, used Conley index techniques to find again the formula written above for the sequence of indices of positive iterates, and to state an analogous formula in the orientation-reversing case: there exist $e \in\{-1,0,1\}$ and $r\geq 0$ such that
$$\{i(f^n,p)\}_{n\geq 1}=e \sigma_1-r\sigma_2,$$ (the same approach was taken, at least in the orientation-preserving case, in the unpublished article \cite{lecalvezyoccozconley}).

In dimension $3$, one can construct, for every $l\in\mathds{Z}$, a local orientation-preserving or reversing homeomorphism having a fixed point $p$ such that $i(f,p)=l$. In case where the fixed point is an isolated invariant set, it is known that the sequence of fixed point indices
$\{i(f^n, p)\}_{n \ge 1}$ is periodic, see \cite{periodicindices}. No further restrictions appear in the
orientation-preserving case. However, it seemed plausible that the orientation-reversing case would exhibit
some particular behavior connected to the planar orientation-preserving case. This was the motivation for this work.
 Corollary \ref{cor:indexinequality} and, more precisely, Theorem C show
the extra restrictions which appear in the orientation-reversing case in dimension 3. In particular, the index is always
less than or equal to 1, as occurred for orientation-preserving planar homeomorphisms. Nevertheless, the
behavior of the sequence of fixed point indices of the iterates is different in both cases, it is much more rigid
in the planar case. Indeed, in dimension $2$, the map $\varphi$ that appear in Theorem \ref{thm:key} will have to preserve or reverse a cyclic order depending whether $f$  preserves or reverses the orientation. In the first case, all periodic points have the same period, in the second case $\varphi$ has at most two fixed points and the other periodic points of $\varphi$ have period $2$.

In \cite{periodicindices}, an example of a fixed point isolated as invariant set of a local homeomorphism of $\mathds{R}^4$
with an unbounded sequence of fixed point indices is shown. Therefore, in higher dimensions the only restriction
one may expect the fixed point sequence to satisfy are Dold's congruences.
Here, we present an orientation-reversing analogue of the example constructed
in Remark 6 of the aforementioned article.
Suppose that $T$ is a 2-torus embedded in $S^3$ and which cuts it in two solid tori, $T^+$ and $T^-$.
Define an orientation-reversing diffeomorphism $h$ of $S^3$ such that the maps
$h : T^+ \to T^+$ and $h^{-1} : T^- \to T^-$ are solenoidal maps of degree $-m$, $m$ a positive integer, which
means conjugate to a mapping
$$ (\theta, z) \mapsto \left(\theta^{-m}, \frac{1}{2}\theta + \frac{1}{2^{m+1}} z\right) $$
defined on the filled torus $\{ \theta \in \mathds{C}, |\theta| = 1\} \times  \{z \in \mathds{C}, |z| \le 1\}$.
The pair $(T^+, T^-)$ is an attractor/repeller decomposition of $S^3$ such that
$$i(h^n, T^+) = 1 + (-m)^n.$$
At the end of Subsection \ref{subsec:proofradial} we showed that given a regular attractor/repeller decomposition
of the sphere, it is possible to define $g : S^3 \to \mathds{R}$ so that the diffeomorphism
$f : S^4 \to S^4$ induced by the skew-product of $g$ and $h$ satisfies:
\begin{itemize}
\item $e^-$ is an isolated fixed point.
\item Its index by $f^n$ is $i(f^n, e^-) = i(h^n, T^+) = 1 + (-m)^n$.
\end{itemize}
Consequently, the sequence of indices is unbounded either from below or from above.

Another important topic concerning fixed point index which attracted some attention was the local study
of conservative homeomorphisms. If we drop the hypothesis about isolation of the fixed point, it remains true
for the planar case that the index of a fixed point is less than or equal to 1 Provided that the
homeomorphism is area-preserving. This result was proved by Pelikan and Slaminka in \cite{pelikanslaminka}.
Previously, Nikishin and Simon had addressed the same question for diffeomorphisms, see \cite{nikishin} and \cite{simon}.
However, the analogue statement referred to orientation-reversing local homeomorphisms
of $\mathds{R}^3$ does not hold, as the following example shows.

Let $l$ be a positive integer and define $f(z) = z + z^l$, a local diffeomorphism of the complex plane in a neighborhood $U$ of the origin. Using the definition, one may check that the fixed point index
of the origin is $l$. Then, the map
$$g: (z, t) \mapsto \left(f(z), \frac{-t}{\vert f'(z)\vert}\right)$$
defined in $U \times [-1, 1]$ is a local diffeomorphism of $\mathds{R}^3$ using the usual identification of $\mathds{C}\sim\mathds{R}^2$. The Jacobian of $g$
is equal to 1, hence it preserves volume.
Since $f$ preserves orientation, $g$ is orientation-reversing. The origin, 0, is the unique fixed point
of $g$ and its index can be computed easily because $g$ is isotopic to the map $(z,t) \mapsto (f(z), -t)$
through an isotopy which does not create extra fixed points in $U \times [-1,1]$.
Then, $$i(g, 0) = i(f, 0) \cdot i(s, 0) = l,$$ where $$s = - \mathrm{id} : [-1, 1] \to [-1, 1].$$

\bigskip
\medskip
\begin{minipage}[t]{0.5\textwidth}
Luis Hernández Corbato\\
Facultad de Matemáticas\\
Universidad Complutense de Madrid\\
Plaza de Ciencias 3\\
28040 Madrid, Spain\\
Email address: \verb"luishcorbato@mat.ucm.es"\\

\medskip
Francisco R. Ruiz del Portal\\
Facultad de Matemáticas\\
Universidad Complutense de Madrid\\
Plaza de Ciencias 3\\
28040 Madrid, Spain\\
Email address: \verb"francisco_romero@mat.ucm.es"
\end{minipage}
\begin{minipage}[t]{0.5\textwidth}
Patrice Le Calvez\\
Institut de Mathématiques de Jussieu\\
UMR CNRS 7586\\
Université Pierre et Marie Curie\\
Case 247, 4, place Jussieu\\
75252 Paris, France\\
Email address: \verb"lecalvez@math.jussieu.fr"
\end{minipage}
\end{document}